\documentclass[11pt]{article}
\usepackage{t1enc}
\usepackage{amsmath,amssymb,color,enumitem}
\usepackage{authblk}
\usepackage{xcolor}

\usepackage{hyperref}

\newtheorem{Thm}{Theorem}
\newtheorem{Def}{Definition}
\newtheorem{Prop}{Proposition}

\newtheorem{Lem}{Lemma}
\newtheorem{Coro}{Corollary}
\newtheorem{Rem}{Remark}

\newenvironment{proof}[1][Proof]{\textbf{#1.} }{\hfill $\square$}

\def \eps{\varepsilon}

\newcommand*{\cE}{\mathcal{E}}
\newcommand*{\cF}{\mathcal{F}}

\newcommand*{\cN}{\mathcal{N}}

\newcommand{\bD}{\mathbb{D}}
\newcommand{\bE}{\mathbb{E}}
\newcommand{\bF}{\mathbb{F}}

\newcommand{\bH}{\mathbb{H}}

\newcommand{\bL}{\mathbb{L}}

\newcommand{\bP}{\mathbb{P}}
\newcommand{\bQ}{\mathbb{Q}}
\newcommand{\bR}{\mathbb{R}}
\newcommand{\bS}{\mathbb{S}}

\newcommand{\mfd}{\mathfrak{d}}

\newcommand{\etamax}{\eta^\star}
\newcommand{\etamin}{\eta_\star}

\newcommand{\lambdamax}{\|\lambda\|}

\newcommand{\psius}{ \psi^\star}

\newcommand{\phius}{ \phi^\star}

\newcommand{\derivX}{\overset{ \text{{\Huge .}}}{X}}

\begin{document}

\title{Asymptotic approach for backward stochastic differential equation with singular terminal condition\thanks{This work began during the visit of P. Graewe to Le Mans. His stay was financially supported by the F\'ed\'eration de Recherche Math\'ematique des Pays de Loire FR CNRS 2962.}}

\author[1]{Paulwin Graewe\thanks{pgraewe@deloitte.de}}
\author[2]{Alexandre Popier\thanks{alexandre.popier@univ-lemans.fr}}
\affil[1]{\small Deloitte Consulting GmbH, Kurf\"urstendamm 23, 10719 Berlin, Germany}
\affil[2]{\small Laboratoire Manceau de Math\'ematiques, Le Mans Universit\'e, Avenue O.~Messiaen, 72085~Le~Mans cedex 9, France}

\date{\today}

\maketitle

\begin{abstract}
In this paper, we provide a one-to-one correspondence between the solution $Y$ of a BSDE with singular terminal condition and the solution $H$ of a BSDE with singular generator. This result provides the precise asymptotic behavior of $Y$ close to the final time and enlarges the uniqueness result to a wider class of generators. 
\end{abstract}

\vspace{0.5cm}


\section{Introduction}

This paper is devoted to the study of the {\it asymptotic behavior} of the solution of backward stochastic differential equations (BSDEs) with {\it singular} terminal condition. We adopt from  \cite{krus:popi:16b} and \cite{popi:06} the notion of a weak (super) solution $(Y,Z)$ to a BSDE of the following form
\begin{equation} \label{eq:sing_BSDE}
-dY_t = \frac{1}{\eta_t} f(Y_t) dt + \lambda_t dt - Z_t dW_t
\end{equation}
where $W$ is a $d$-dimensional Brownian motion on a probability space $(\Omega,\cF,\bP)$ with a filtration $\bF = (\cF_t)_{t\geq 0}$. The filtration $\bF$ is the natural filtration generated by $W$ and is supposed to be complete and right continuous. 
The particularity here is that we allow the {\it terminal condition} $\xi$ to be {\it singular}, in the sense that $\xi = +\infty$ a.s. 

Since the seminal paper by Pardoux and Peng \cite{pard:peng:90} BSDEs have proved to be a powerful tool to solve stochastic optimal control problems (see e.g.\ the survey article \cite{el1997backward} or the books \cite{pham2009continuous,yong:zhou:99}). BSDEs with singular terminal condition provide a purely probabilistic solution of a stochastic control problem with a terminal constraint on the controlled process. The analysis of optimal control problems with state constraints on the terminal value is motivated by models of optimal portfolio liquidation under stochastic price impact. The traditional assumption that all trades can be settled without impact on market dynamics is not always appropriate when investors need to close large positions over short time periods. In recent years models of optimal portfolio liquidation have been widely developed, see, e.g. \cite{almg:12,almg:chri:01,fors:kenn:12,gath:schi:11,hors:nauj:14,krat:scho:13}, among many others. In \cite{anki:jean:krus:13}, the following problem is considered: minimizing the cost functional 
\begin{equation}\label{eq:control_pb_intro}
J(X) = \bE \left[  \int_0^T \left( [(p-1)\eta_s]^{p-1} |\derivX_s|^p + \lambda_s |X_s|^p \right) ds  \right]
\end{equation}
over all progressively measurable processes $X$ that satisfy the dynamics
\begin{equation*}
X_s =x +\int_0^s \derivX_u du 
\end{equation*}
with the terminal constraint that $X_T = 0$ a.s. Here $p>1$ and the processes $\eta$ and $\lambda$ are non-negative and progressively measurable. In this framework the state process $X$ denotes the agent's position in the financial market. At each point in time $t$ she can trade in the primary venue at a rate $\derivX_t$ which generates costs $[(p-1)\eta_t]^{p-1} |\derivX_t|^p$ incurred by the stochastic price impact parameter $\eta$. The term $\lambda_t |X_t|^p$ can be understood as a measure of risk associated to the open position. $J(X)$ thus represents the overall expected costs for closing an initial position $x$ over the time period $[0,T]$ using strategy $X$. In \cite{anki:jean:krus:13}, optimal strategies and the value function of this control problem \eqref{eq:control_pb_intro} are characterized with the BSDE
\begin{equation} \label{eq:bsde_contr_prob}
-dY_t  = - \frac{Y_t|Y_t|^{q}}{\eta_t} dt  + \lambda_t dt - Z_t d W_t
\end{equation}
with $\displaystyle \lim_{t\to T} Y_t =+\infty$. Here $q+1$ is the H\"older conjugate of $p$. The function $f:y\mapsto -y|y|^{q}$ is here a polynomial function. Variants of the position targeting problem \eqref{eq:control_pb_intro} have been studied in \cite{anki:krus:13b,grae:hors:qiu:13,grae:hors:sere:18,schie:13}. Note that these problems are particular cases of the stochastic calculus of variations (see \cite{anki:from:krus:18}). 

\bigskip
Let us explain the methodology to obtain a solution for the BSDE \eqref{eq:bsde_contr_prob}. The most common approach in the literature is the so-called \textit{penalization approach}, see, e.g., \cite{popi:06}, \cite{popi:07}, \cite{anki:jean:krus:13}, \cite{grae:hors:qiu:13}, \cite{krus:popi:16b}, and the references therein.  The idea of the penalization approach is to relax the binding liquidation constraint by penalizing open position in the underlying liquidation problem. In \cite{anki:jean:krus:13}, as in \cite{krus:popi:16b} for more general driver, the authors use the penalization approach, replacing the singular terminal by a constant $n$ and letting $n$ go to $+\infty$. The convergence is obtained by a comparison principle for solution of BSDEs (see \cite{krus:popi:16} or \cite{pard:rasc:14}). In \cite{grae:hors:sere:18}, the approach consists in the study of the precise asymptotic behavior at time $T$ of the solution $Y$ of~\eqref{eq:bsde_contr_prob}. Roughly speaking, the major singular term of $Y$ is then removed to obtain a non-singular problem. The key of this \textit{asymptotic approach} is to establish
sharp a priori estimates of the singular solution at the terminal time. In \cite{grae:hors:sere:18}, the authors consider a time-homogeneous Markov setting and obtain the a priori estimates and uniqueness by establishing a general comparison principle for singular viscosity solution to \eqref{eq:bsde_contr_prob}. These results are based on time-shifting arguments, applied similar before in \cite{popi:06}, which in general do not apply in a non-time-homogeneous setting. However, it is outlined in \cite{grae:hors:sere:18} how the shifting argument may be applied in non-Markov settings to obtain sharp a priori estimates of the singular solution to \eqref{eq:bsde_contr_prob}. One major result of \cite{grae:hors:sere:18} is the uniqueness of the solution of \eqref{eq:bsde_contr_prob}, under boundedness assumptions on the coefficients $\mu$ and $\lambda$ (see also \cite{grae:hors:qiu:13}). This approach is also used in the recent paper \cite[Section 4]{horst:xia:18} to obtain uniqueness without the boundedness condition.

\bigskip

Let us outline in which directions our findings generalize some results from these papers. In \cite{krus:popi:16b} the generator may depend also on $Z$ in a non trivial way; here our generator has a special form. However in the previously mentioned papers (see \cite{krus:popi:16b}) $f$ is assumed to be a polynomial function or of polynomial growth w.r.t. $y$, that is $f(y) \leq -y|y|^q$. In the rest of the paper we assume that 
\begin{enumerate}[label=\textbf{(A\arabic*)}]
\item \label{A1} {\it There exist three constants $0 < \etamin < \etamax$ and $\lambdamax \geq 0$ such that a.s. for any $t\in [0,T]$ 
$$\etamin \leq \eta_t \leq \etamax,\qquad 0 \leq \lambda_t \leq \lambdamax.$$}
\item \label{A2} {\it The function $f$ is continuous and non increasing, with $f(0)=0$ and with continuous derivative.}
\item \label{A3} {\it For any $x > 0$, the function
\begin{equation} \label{eq:def_G}
G(x):=\int_x^\infty \frac{1}{-f(t)}\,dt
\end{equation}
is well-defined on $(0,\infty)$.}
\end{enumerate}
Note that if $\eta$ and $\lambda$ are deterministic, the BSDE \eqref{eq:sing_BSDE} becomes an ODE and this condition \ref{A3} is necessary and sufficient to ensure that the solution can be equal to $+\infty$ at time $T$, but finite at any time $t < T$. Under these assumptions, we prove {\it existence of a minimal solution} $(Y,Z^Y)$ of the BSDE \eqref{eq:sing_BSDE} (Proposition \ref{prop:exist_min_sol}). The function $f(y)=-(y+1)|\log(y+1)|^q$ is an example satisfying \ref{A3} but not covered by the preceding papers. 

\medskip
Our main result concerns the {\it asymptotic behaviour and uniqueness of this minimal solution} (Theorem \ref{thm:uniq_sing_BSDE}). Under some (sufficient) conditions \ref{C1_concave} (concavity and regularity), \ref{C2_tech} to \ref{C4_tech} (behaviour near $+\infty$) on $f$ and under the assumption \ref{H_diff} on $\eta$ (It\^o's process), we prove that $Y$ is equal to: a.s. 
\begin{equation} \label{eq:asymp_beha_Y}
Y_t = \phi \left(A_t \right) - \phi'\left( A_t \right) H_t,\quad \forall t \in [0,T], 
\end{equation}
where  
\begin{itemize}
\item for any $t\in [0,T]$
\begin{equation} \label{eq:def_A}
A_t = \bE \left[  \int_t^T \frac{1}{\eta_s} ds \bigg| \cF_t \right],
\end{equation}
\item $\phi=G^{-1}$ solves the ODE: $\phi' = f\circ \phi$ with initial condition $\phi(0)=+\infty$.
\item The process $H$ is the {\it unique solution} of a BSDE with terminal condition 0 and with a {\it singular generator} $F^H$ in the sense of \cite{jean:reve:14}:
\begin{equation}\label{eq:BSDE_H}
 H_t = \int_t^T F^H(s,H_s,Z^H_s) ds + \int_t^T  Z^H_s dW_s.
\end{equation}
\end{itemize}
Singular generator means that $\int_t^T |F^H(s,h,z)| ds = +\infty$. One singularity comes from the explosion at time $T$ and creates trouble only close to time $T$. But in $F^H$ there are linear terms including the martingale part of the process $A$. These terms have to be controlled on the whole interval $[0,T]$, and not only on a neighborhood of $T$. This is the reason of the additional assumptions \ref{C1_concave} to \ref{C4_tech} and \ref{H_diff}. Let us emphasize that there were only two results about uniqueness, namely \cite[Theorem 6.3]{grae:hors:sere:18} and \cite[Theorem 4.3]{horst:xia:18} for the power case $f(y)=-y|y|^q$, $q>0$. Uniqueness was proved by showing that any solution $(Y,Z^Y)$ is the value function of the control problem \eqref{eq:control_pb_intro}. Here the proof is based on the comparison principle for BSDEs.

The {\it asymptotic behavior} follows from the boundedness of the process $-\phi'(A) H / \phi(A)$ on some deterministic neighborhood of $T$: a.s.
$$\phi(A_t) = \phi\left( \bE \left[  \int_t^T \frac{1}{\eta_s} ds \bigg| \cF_t \right]\right) \leq Y_t \leq (1+\kappa) \phi\left( \bE \left[  \int_t^T \frac{1}{\eta_s} ds \bigg| \cF_t \right]\right),$$
where the constant $\kappa$ depends on the coefficients $\eta$, $\lambda$ and $f$. This inequality justifies the decomposition \eqref{eq:asymp_beha_Y}: the dominating term at time $T$ is $\phi(A_t)$. Moreover if $\eta$ is deterministic, the solution of the BSDE \eqref{eq:sing_BSDE} with $\lambda\equiv 0$ is exactly $\phi(A)$. Somehow the randomness and the presence of $\lambda$ are considered as a perturbation of $\phi(A)$ and the decomposition \eqref{eq:asymp_beha_Y}  is a kind of Taylor's expansion of $Y$. Note that different decompositions could be considered (see the end of Section \ref{sect:asymp_beha} or the extensions developed in Section \ref{sect:general_gene}); they all give the same asymptotic behavior for $Y$ at time $T$.

\medskip
As a consequence, we provide a new {\it one-to-one correspondence} between the BSDE \eqref{eq:sing_BSDE} with singular terminal condition and the BSDE \eqref{eq:BSDE_H} with singular generator. The main drawback for BSDE with singular terminal condition is the lack of approximation scheme with some rate of convergence. Indeed most of numerical schemes for BSDE are based on backward induction starting at the terminal value. The correspondence \eqref{eq:asymp_beha_Y} between $Y$ and $H$ could be a promising solution for numerical scheme, since the terminal value of $H$ is zero. The singularity of the generator of $H$ is a serious obstacle. However the works made on irregular terminal condition (see \cite{geis:geis:gobe:12} or \cite{gobe:makh:10}) could be inspiring to construct numerical schemes on the BSDE \eqref{eq:BSDE_H}. This point is left for further research.

\bigskip
We also study two different cases. The first one is the power case $f(y) = -y|y|^q$. If $\eta$ is an It\^o process, we follow the arguments of the paper \cite{grae:hors:sere:18} for a PDE and show that the relation \eqref{eq:asymp_beha_Y} can be formulated as follows: a.s.
\begin{equation} \label{eq:asymp_beha_power_case}
\forall t \in [0,T],\quad Y_t= \left( \eta_t \right)^{1/q} \phi(T-t)-\phi'(T-t) H^{\sharp}_t ,
\end{equation}
where $H^\sharp$ solves the BSDE with a singular generator $F^\sharp$:
$$H^{\sharp}_t= \int_t^T  F^\sharp(s,H^\sharp_s)\,ds-\int_t^T Z^{H^\sharp}_sdW_s,$$
see Theorem \ref{thm:power_case}. The interesting point here is that the unique solution $H^\sharp$ is {\it constructed by the Picard iteration} in the space  
\begin{equation} \label{eq:def_cH_delta} 
\mathcal H^\delta:=\{H\in L^\infty(\Omega;C([T-\delta, T];\mathbb R)):\|H\|_{\mathcal H^\delta}<+\infty\}
\end{equation}
 endowed with the weighted norm 
\begin{equation} \label{eq:def_norm_cH_delta}
\|H\|_{\mathcal H^\delta}=\left\|\sup_{t\in{[T-\delta,T)}}(T-t)^{-2}|H_t|\right\|_\infty.
\end{equation}
The constant $\delta$ depends on $q$ and the norm of the coefficients $\eta$ and $\lambda$. This construction has two main advantages: first we have a more accurate behavior of $H^\sharp$ at time $T$, secondly this construction is more tractable for numerical approximation. Since $H^\sharp$ is obtained by a fixed point argument in a weighted space, we strongly believe that it could be a way to compute $H^\sharp$, and thus $Y$.

\medskip
The second case is an attempt to relax the assumptions on $f$ (in particular concavity) and on $\eta$ (diffusion process). Instead of \eqref{eq:asymp_beha_Y}, we look at: a.s.
\begin{equation} \label{eq:asymp_beha_general_case}
\forall t \in [0,T],\quad Y_t = \phi(A_t) - \phi' \left(\frac{T-t}{\etamax} \right)\widehat H_t =\phi(A_t) - \psius(t) \widehat H_t.
\end{equation}
It is important to note that there is some asymmetry in \eqref{eq:asymp_beha_general_case} since the first term with $\phi$ is random, whereas the second with $\psius$ is deterministic. This method avoids assuming extra assumptions on $f$ but we loose uniqueness of the solution (see Theorem \ref{thm:one_one_correspondance}).

\bigskip

The paper is decomposed as follows.  In Section \ref{sect:sol_sing_BSDE}, we recall and extend several results concerning the {\it existence of the solution $(Y,Z^Y)$ of the BSDE \eqref{eq:sing_BSDE} with singular terminal condition} $+\infty$ and provide some a priori estimates on this solution $Y$ and on $Z^Y$. 

Section \ref{sect:asymp_beha} is dedicated to the decomposition \eqref{eq:asymp_beha_Y} and to the statement of Theorem \ref{thm:uniq_sing_BSDE}. We explain our additional assumptions on the coefficients $\eta$ and $f$ of the BSDE \eqref{eq:sing_BSDE}. In the subsection \ref{ssect:examples}, the reader finds here several examples of functions $f$, for which these hypotheses hold. In the part \ref{ssect:generator_f_H} we study the properties of the coefficients of the generator $F^H$, and in the next subsection \ref{ssect:proof_main_thm} we prove the {\it existence and uniqueness of the solution of the BSDE  with a singular generator and with terminal condition 0} (Equation \eqref{eq:BSDE_H})  and the {\it one-to-one correspondence} between the solutions $Y$ and $H$. The rest of the section (part \ref{ssect:asymp_behaviour}) concerns the asymptotic behaviour of $Y$. 

In Section \ref{sect:general_gene}, we study the power case and we prove that $H^\sharp$ can be constructed by a Picard approximation scheme. And we briefly explain how to handle \eqref{eq:asymp_beha_general_case} without extra conditions on $f$. 

Finally in the appendix, the non-negativity assumption of $\lambda$ is removed and we give a {\it self-contained construction of the solution} $(H,Z^H)$ (without any reference to $Y$) which extends the existence result of \cite{jean:reve:14}. 

 %

In the continuation, unimportant constants will be denoted by $C$ and they could vary from line to line.

\section{BSDEs with singular terminal condition} \label{sect:sol_sing_BSDE}

Let us introduce the following spaces for $p\geq 1$.
\begin{itemize}
\item $\bD^p(0,T)$ is the space of all adapted c\`adl\`ag\footnote{French acronym for right-continuous with left-limit} processes $X$ such that
$$\bE \left(  \sup_{t\in [0,T]} |X_t|^p \right) < +\infty.$$
\item $\bH^p(0,T)$ is the subspace of all predictable processes $X$ such that
$$\bE \left[ \left( \int_0^T |X_t|^2 dt\right)^{\frac{p}{2}} \ \right] < +\infty.$$
\item $\bS^p(0,T) = \bD^p(0,T) \times \bH^p(0,T)$ and $\bS^\infty(0,T) = \bigcap_{p\geq 1} \bS^p(0,T)$. 
\end{itemize}
Whenever the notation $T-$ appears in the definition of a process space, we mean the set of all processes whose restrictions satisfy the respective property when $T-$ is replaced by any $T-\eps$, $\eps > 0$. 
Let us define the notion of solution for a singular BSDE. 

\begin{Def}[BSDE with singular terminal condition] \label{def:sing_cond_sol}
The process $(Y,Z^Y)$ is a solution of the BSDE \eqref{eq:sing_BSDE} with terminal condition $+\infty$ if:
\begin{itemize}
\item There exists $p>1$ such that $(Y,Z^Y)\in \bS^p(0,T-)$.
\item For any $0\leq s \leq t < T$, 
$$Y_s = Y_t + \int_s^t \left[ \frac{1}{\eta_u} f(Y_u) + \lambda_u \right] du - \int_s^t Z^Y_u dW_u.$$
\item $Y_t \geq 0$ for any $t \in [0,T]$,a.s.  .
\item A.s. $\displaystyle \lim_{t\to T} Y_t =+\infty$.
\end{itemize}
\end{Def}
Minimality of the solution $(Y,Z^Y)$ means that for any other process $(\widetilde Y,\widetilde Z)$ satisfying the previous four items, a.s. $\widetilde Y_t \geq Y_t$ for any $t$.

In the BSDE \eqref{eq:sing_BSDE}, the generator is of the form:
$$(\omega,t,y) \mapsto \frac{1}{\eta_t(\omega)} f(y) + \lambda_t(\omega).$$
Recall that here and in the rest of the paper, the conditions \ref{A1} and \ref{A2} hold. Supposing that $f$ is continuous and non increasing\footnote{For monotone BSDEs, the classical assumption is: for some $\mu \in \bR$ and any $(y,y')\in \bR^2$, $(f(y)-f(y'))(y-y') \leq \mu (y-y')^2$. By a very standard transformation (see \cite[Proof of Corollary 5.26]{pard:rasc:14}) we may assume w.l.o.g. that $\mu = 0$, thus $f$ is non increasing.}  is coherent with the existence and uniqueness results concerning monotone BSDEs (see \cite[Chapter 5.3.4]{pard:rasc:14}). Note that if $f(0)\neq 0$, then
$$\frac{f(y)}{\eta_t} + \lambda_t = \frac{f(y)-f(0)}{\eta_t} + \lambda_t + \frac{f(0)}{\eta_t} = \frac{\widetilde f(y)}{\eta_t} + \widetilde \lambda_t,$$
provided that $\widetilde \lambda_t \geq 0$. The non-negativity of $\lambda$ is natural for the control problem \eqref{eq:control_pb_intro} and leads to a more accurate expansion of $Y$. However it is not necessary (see Section \ref{ssect:sign_lambda} in the appendix for a short discussion on this point). Somehow and to summarize, the only stronger condition on our type of generator is the regularity: $f \in C^1(\bR)$.

Let us summarize the known results on the singular BSDE \eqref{eq:sing_BSDE}, assuming that the assumptions \ref{A1} and \ref{A2} hold.
\begin{itemize}
\item From \cite[Theorem 1]{krus:popi:16b}, if $f(y)\leq -y|y|^q$ for some $q>0$, the BSDE \eqref{eq:sing_BSDE} has a minimal solution $(Y,Z^Y)$.
\item From \cite[Theorem 6.3]{grae:hors:sere:18}, if $f(y) = - y|y|^q$, the solution is unique. 
\end{itemize}
To obtain the existence of a solution, a key step is the existence of an a priori estimate on $Y$ (see \cite{anki:jean:krus:13}, \cite{grae:hors:sere:18} or \cite{krus:popi:16b}). Using the arguments of \cite[Proposition 6.1]{grae:hors:sere:18}, we obtain that if $f(y)\leq -y|y|^q$, any solution of the BSDE \eqref{eq:sing_BSDE} satisfies:
\begin{equation} \label{eq:power_case_a_priori_estim}
Y_t  \leq  \frac{1}{(T-t)^{p}} \bE \left[ \int_t^T\left(  \left( \frac{\eta_s}{q} \right)^{\frac{1}{q}} + (T-s)^{p} \lambda_s \right) ds \bigg| \cF_t\right] 
\end{equation}
where $p$ is the H\"older conjugate of $q+1$. 


The goal of this section is to extend the existence result to a larger class of functions $f$. Let us consider the ordinary differential equation (ODE): $y'=-f(y)$ with the terminal condition $y(T) = +\infty$. There exists a solution if and only if the function $G$ given by \eqref{eq:def_G} is well-defined at least on some interval $(\kappa,+\infty)$, with $\kappa=\sup \{y \geq 0, \ f(y) = 0\}$, meaning that $f\equiv 0$ on $[0,\kappa]$. Note that the function $G$ is positive, strictly decreasing and convex, such that $G(\infty)=0$ and the smoothness of $f$ implies that $G(\kappa) = + \infty$. Then the solution $y$ is given by: $y(t) = G^{-1}(T-t)$ on $[0,T]$. Defining $f^\kappa(x)=f(x+\kappa)$, $G^\kappa(x) = G(x+\kappa)$ yields that $G^\kappa$ is defined on $(0,+\infty)$. Moreover the solution $y$ is given by: $y(t) =G^{-1}(T-t)= (G^\kappa)^{-1}(T-t) - \kappa$ and solves $y' = -f^\kappa(y)$ together with $y(T) = +\infty$. 

Hence w.l.o.g. we assume from now on that \ref{A3} holds. Recall that $\phi=G^{-1}$ is decreasing and $C^2$ on $(0,\infty)$ and solves $\phi' = f\circ \phi$. 
\begin{Rem}
In the inequality \eqref{eq:power_case_a_priori_estim}, the boundedness assumption {\rm \ref{A1}} can be replaced by some integrability condition on $\lambda$ and $\eta$. Thereby weaker conditions are considered in \cite{anki:jean:krus:13} or \cite{krus:popi:16b} to obtain existence and in \cite{horst:xia:18} for the uniqueness. However under {\rm \ref{A3}}, the required integrability condition is not easy to determine. For small value of $q$, \eqref{eq:power_case_a_priori_estim} is finite if $\eta$ is in $\bL^{1/q}(\Omega \times [0,T])$, that is not so far from boundedness. 
\end{Rem}

Let us now give an upper bound on $Y$, similar to \cite[Proposition 6.1]{grae:hors:sere:18} but again for a general function $f$. Let us consider the function 
$$\widetilde G(x) = \int_x^\infty \frac{-1}{ \lambdamax + \frac{f(y)}{\etamax}}  dy = \etamax  \int_x^\infty \frac{1}{- \lambdamax \etamax -f(y)} dy$$
defined on the interval $I=(f^{-1}(- \lambdamax \etamax),+\infty)$. Since $f$ is a function with continuous derivative, $\widetilde G(f^{-1}(- \lambdamax \etamax))=+\infty$. If we define $\vartheta = \widetilde G^{-1}$, this function is well-defined on $(0,+\infty)$, with $\vartheta(0)=+\infty$ and satisfies:
\begin{equation}\label{eq:upper_ODE}
\vartheta' = \lambdamax + \frac{f(\vartheta)}{\etamax}.
\end{equation}
Note that the function $\vartheta$ strongly depends on $\etamax$, $\lambdamax$ and $f$. 
\begin{Lem} \label{lem:gene_a_priori_estim}
Assume that the process $(U,Z^U)$ satisfies the dynamics: for any $\eps > 0$ and $0\leq t \leq T-\eps$
\begin{equation} \label{eq:a_priori_upper_bound_dyn}
U_t + \zeta_t = U_{T-\eps}+\int_t^{T-\eps} \left[ \lambda_s + \frac{1}{\eta_s} f(U_s) \right]ds - \int_t^{T-\eps} \Theta_s ds  -  \int_t^{T-\eps}  Z^U_s dW_s,
\end{equation}
where $\zeta$ and $\Theta$ are two non-negative processes. Then a.s. for all $t \in [0,T)$,
\begin{equation} \label{eq:a_priori_estimate_U}
0 \leq U_t \leq \vartheta(T-t).
\end{equation}
\end{Lem}
\begin{proof}
We proceed as in the proof of \cite[Proposition 6.1]{grae:hors:sere:18}, namely we shift the singularity. 
Take any $0< \eps$ such that $0 \leq T-\eps < T$. 
The function $( \vartheta \left( T-\eps -t \right), \ t \in [0,T-\eps]) $ solves the ODE: $y(T-\theta)=+\infty$ and 
$$y' = -\lambdamax - \frac{f(y)}{\etamax}.$$
By the comparison principle again we have that $U_t \leq \vartheta \left( T-\eps -t \right)$ on $[0,T-\eps]$. Since $U$ does not depend on $\eps$, we obtain that a.s. 
$$\forall t \in [0,T), \quad U_t \leq \vartheta \left( T -t \right) .$$
This achieves the proof of the lemma.
\end{proof}

As a by-product, our proof implies that for any non-negative solution $(Y,Z^Y)$ of the BSDE \eqref{eq:sing_BSDE}, we have a.s. on $[0,T]$:
\begin{equation} \label{eq:a_priori_estim_Y}
Y_t \leq  \vartheta(T-t).
\end{equation}
 Compared to \eqref{eq:power_case_a_priori_estim}, in the power case $f(y)=-y|y|^q$, this estimate is less accurate. However it holds for functions without polynomial growth. 
A typical example is $f(y) = - (y+1)|\log(y+1)|^q$ for some $q > 1$. 
\begin{Prop} \label{prop:exist_min_sol}
Under Conditions {\rm \ref{A1}},  {\rm \ref{A2}} and {\rm \ref{A3}}, the BSDE \eqref{eq:sing_BSDE} has a minimal non-negative solution $(Y,Z^Y) \in \bS^\infty(0,T-)$ such that \eqref{eq:a_priori_estim_Y} holds. 
\end{Prop}
\begin{proof}
The existence of a non-negative solution can be obtained by the same penalization arguments as in \cite{anki:jean:krus:13} or \cite{krus:popi:16b}. We use the a priori estimate \eqref{eq:a_priori_estim_Y} in order to obtain the convergence of the penalization scheme on any interval $[0,T-\eps]$. Minimality can be proved as in \cite[Proposition 4]{krus:popi:16b}. Thus we skip the details here.
\end{proof}

\begin{Rem}[Generator depending on $Z$] \label{rem:Z_gene}
Assume that the generator has the form:
$$(t,\omega,y,z) \mapsto \frac{f(y)}{\eta_t(\omega)} + \lambda_t (\omega)+ \zeta(t,\omega,z),$$
where there exists a constant $C$ such that for any $(t,\omega,z,z')$ 
$$0\leq \zeta(t,\omega,0)\leq C,\qquad |\zeta(t,\omega,z) - \zeta(t,\omega,z')| \leq C|z-z'|.$$
Using the Girsanov theorem, existence of a solution can be derived directly from Proposition \ref{prop:exist_min_sol}. Indeed if $(Y,Z^Y)$ is a solution of the BSDE \eqref{eq:sing_BSDE}, then 
\begin{align*}
Y_t & = Y_u +\int_t^u \left[ \frac{f(Y_s)}{\eta_s} + \lambda_s + \zeta(s,0) \right] ds - \int_t^u Z^Y_s \left[ dW_s - \zeta^{Z^Y}_s ds\right] \\
& = Y_u +\int_t^u \left[ \frac{f(Y_s)}{\eta_s} + \lambda_s + \zeta(s,0) \right] ds - \int_t^u Z^Y_s  d\widetilde W_s
\end{align*}
where the process  $\zeta^{Z^Y}$ is bounded by $C$ and $\widetilde W$ is a Brownian motion under a probabilty measure $\bQ$ equivalent to $\bP$. 
Moreover all results in this paper remain valid under this probability measure $\mathbb Q$ equivalent to $\mathbb P$. 
\end{Rem}

We define the process $A$ by \eqref{eq:def_A}. 
Note that from the boundedness of $\eta$
\begin{equation} \label{eq:bound_A}
 \frac{1}{\etamax }(T-t) \leq A_t \leq \frac{1}{\etamin} (T-t).
 \end{equation}
Let us add a lower bound on $Y$, similar to \cite[Estimate 3.7]{anki:jean:krus:13} (see also \cite[Proposition 4.1]{horst:xia:18}), but for a more general driver $f$. 
\begin{Lem}\label{lem:lower_bound}
The minimal solution $Y$ satisfies a.s. for any $t \in [0,T]$,
\begin{equation} \label{eq:lower_Y_estimate}
Y_t \geq  \phi\left(A_t  \right).
\end{equation}
\end{Lem}
\begin{proof}
Indeed the process $A$ satisfies
\begin{equation} \label{eq:dynamics_A}
-dA_t = \frac{1}{\eta_t} dt + Z^A_t dW_t
\end{equation}
for some $Z^A \in \bH^2(0,T)$. For some $L > 0$, define $A^L_t = \frac{1}{L} + A_t$. 
Since $\phi$ is a smooth function, if $U^L = \phi(A^L)$, It\^o's formula leads to
\begin{eqnarray*}
-dU^L_t &= &\phi'(A^L_t) \left[ \frac{1}{\eta_t } dt + Z^A_t dW_t\right]  - \frac{1}{2} \phi''(A^L_t) (Z^A_t)^2 dt \\
&=& \frac{1}{\eta_t} f(U^L_t) dt- \frac{1}{2} \phi''(A^L_t) (Z^A_t)^2dt  + Z^{U^L} dW_t . 
\end{eqnarray*}
Note that $\phi''(x) = f'(\phi(x)) f(\phi(x)) \geq 0$, thus $\frac{1}{2} \phi''(A^L_t) (Z^A_t)^2 \geq 0$. 
Since $U^L_T = \phi(1/L)$, from the comparison principle for monotone BSDE (see \cite[Proposition 5.34]{pard:rasc:14}) and the construction of $Y$ by approximation, we obtain that $Y_t \geq U^L_t$. Passing through the limit on $L$ leads to the conclusion. 
\end{proof}

To finish this section, let us give an estimate of $Z^Y$. Let us also emphasize that this upper bound is valid for any solution of the BSDE \eqref{eq:sing_BSDE}, since the proof only uses the dynamics on $[0,T)$ and the a priori estimate \eqref{eq:a_priori_estim_Y} on $Y$, but not the construction by penalization of $Y$. In the power case ($f(y)=-y|y|^q$), it is known (see \cite{bank:voss:18,popi:06}) that 
$$\bE \left[ \left(  \int_0^T(T-s)^{2/q}(Z^Y_s)^2 ds \right) \right]   <+\infty.$$
%

\begin{Lem} \label{lem:estim_ZY}
Assume that $f$ is concave. Any solution $(Y,Z^Y)\in \bS^\infty(0,T-)$ of \eqref{eq:sing_BSDE} satisfies for all $p \in [1,+\infty)$
$$\bE \left[ \left(  \int_0^T \dfrac{-f'(Y_s)}{(f(Y_s))^2} (Z^Y_s)^2 ds \right)^p \right]   <+\infty.$$

\end{Lem}
\begin{proof}
From \eqref{eq:lower_Y_estimate}, $Y$ remains bounded away from zero on $[0,T]$. Thus let us apply the function $G$ to $Y$:
\begin{eqnarray*}
G(Y_t) -G(Y_0) & = &  \int_0^t \frac{1}{f(Y_s)} \left(-\frac{1}{\eta_s} f(Y_s) - \lambda_s \right) ds + \int_0^t \frac{1}{f(Y_s)} Z^Y_s dW_s \\
& + &\frac{1}{2} \int_0^t \frac{-f'(Y_s)}{(f(Y_s))^2}(Z^Y_s)^2 ds. 
\end{eqnarray*}
Hence
$$
0\leq \frac{1}{2} \int_0^t \frac{-f'(Y_s)}{(f(Y_s))^2}(Z^Y_s)^2 ds \leq G(Y_t)  +  \int_0^t  \left(\frac{1}{\eta_s}  \right) ds - \int_0^t \frac{1}{f(Y_s)} Z^Y_s dW_s.$$
Now for $p\geq1$, there exists $C_p$ such that 
\begin{eqnarray*}
0\leq \left(  \int_0^t \frac{-f'(Y_s)}{(f(Y_s))^2}(Z^Y_s)^2 ds \right)^p \leq C_p \left( (G(Y_t))^p  + \frac{t^p}{\etamin^p}
 +\sup_{u\in [0,t]} \left| \int_0^u \frac{1}{f(Y_s)} Z^Y_s dW_s \right|^p \right).
\end{eqnarray*}
From Equations \eqref{eq:bound_A} and \eqref{eq:lower_Y_estimate}, since $G$ is non-increasing
$$0\leq G(Y_t) \leq G(\phi(A_t))\leq G(\phi(T/\etamin)).$$
Taking the expectation and using BDG's inequality we obtain for any $t<T$
\begin{eqnarray*}
\bE\left[ \left(  \int_0^t \frac{-f'(Y_s)}{(f(Y_s))^2}(Z^Y_s)^2 ds \right)^p \right] &\leq &C_p \left( (G(\phi(T/\etamin)))^p  + \frac{T^p}{\etamin^p} \right) \\
& + &C_p  \bE  \left[ \left( \int_0^t \frac{1}{(f(Y_s))^2} (Z^Y_s)^2 ds \right)^{p/2} \right].
\end{eqnarray*}
Since $-f'$ is non-decreasing (that is $f$ is concave), for any $s \in [0,T]$
$$0\leq  \frac{1}{-f'(Y_s)} \leq \frac{1}{-f'(\phi(A_s))}\leq \frac{1}{-f'(\phi(T/\etamin))}<+\infty. $$
Therefore for any $t<T$
\begin{eqnarray*}
\bE\left[ \left(  \int_0^t \frac{-f'(Y_s)}{(f(Y_s))^2}(Z^Y_s)^2 ds \right)^p \right] &\leq &C+ C  \bE  \left[ \left( \int_0^t \frac{-f'(Y_s)}{(f(Y_s))^2} (Z^Y_s)^2 ds \right)^{p/2} \right].
\end{eqnarray*}
Using Jensen's inequality if $\varpi = \bE\left[ \left(  \int_0^t \frac{-f'(Y_s)}{(f(Y_s))^2}(Z^Y_s)^2 ds \right)^p \right] $, $\varpi  \leq C+ C \varpi^{1/2}$, hence for any $t< T$
$$\bE \left[ \left(  \int_0^t \frac{-f'(Y_s)}{(f(Y_s))^2}(Z^Y_s)^2 ds \right)^p \right] \leq C. $$
By monotone convergence theorem, we obtain the desired estimate. 
\end{proof}

Using the monotonicity of $f$ and $f'$, using \eqref{eq:a_priori_estim_Y} and \eqref{eq:lower_Y_estimate}, we get
$$  \bE \left[ \left(  \int_0^T \frac{-f'(\phi((T-s)/\etamin))}{ f(\vartheta(T-s))^2}(Z^Y_s)^2 ds \right)^p \right]  <+\infty.$$
In the power case $f(y) = -y|y|^q$,  $\phi(x) =\left(  \dfrac{1}{qx} \right)^{\frac{1}{q}}$. Since $\theta(T-t)$ is equivalent to $(T-t)^{-1/q}$, we have
$$\bE \left[ \left(  \int_0^T(T-s)^{2/q+1}(Z^Y_s)^2 ds \right) \right]   <+\infty.$$
Hence this estimate is not optimal. Nevertheless it is sufficient for our purpose in the proof of Proposition \ref{prop:existence_sol_BSDE_sing_gen} in Section \ref{ssect:proof_main_thm} below.

\section{Asymptotic behaviour: the main result} \label{sect:asymp_beha}

Recall that $A$ is given by \eqref{eq:def_A} and satisfies  \eqref{eq:dynamics_A} and $\phi$ verifies: $\phi' = f\circ \phi$. 
For
\[
-dY_t=\frac{1}{\eta_t} f(Y_t)\,dt+\lambda_t\,dt-Z^Y_t\,dW_t
\]
we make the ansatz \eqref{eq:asymp_beha_Y}, that is $Y_t=\phi(A_t) - \phi'(A_t)H_t.$ Hence we obtain the heuristic dynamics of $H$, namely:
\begin{eqnarray} \nonumber
	-dH_t&=&\frac{1}{\eta_t (-\phi'(A_t))} \left[ f(\phi(A_t)-\phi'(A_t) H_t) - f(\phi(A_t)) +f'(\phi(A_t)) \phi'(A_t) H_t\right] dt \\ \nonumber
	& +& \left[ -\frac{\lambda_t}{\phi'(A_t)} - \frac{\phi^{(2)}(A_t)  A_t }{2\phi'(A_t)} \frac{(Z^A_t)^2}{A_t} \right] dt\\ \nonumber
	& +& \left[ \frac{\phi^{(3)}(A_t) (A_t)^2}{2\phi'(A_t)}  \left( \frac{Z^A_t}{A_t} \right)^2 H_t -\frac{A_t \phi^{(2)}(A_t) }{\phi'(A_t)} \frac{Z^A_t}{A_t} Z^H_t \right] dt -Z^H_t\,dW_t \\ \nonumber
	&=&\frac{1}{\eta_t (-\phi'(A_t))} \left[ f(\phi(A_t)- \phi'(A_t) H_t) - f(\phi(A_t)) + f'(\phi(A_t)) \phi'(A_t) H_t\right]  dt \\ \label{eq:dyn_H_sym_case}
	& + &\left[- \frac{\lambda_t}{\phi'(A_t)} + \frac{\kappa^1_t}{2} \frac{(Z^A_t)^2}{A_t}  +\frac{ \kappa_t^1  \kappa^2_t}{2}  \left( \frac{Z^A_t}{A_t} \right)^2 H_t +  \kappa_t^1 \frac{Z^A_t}{A_t} Z^H_t \right] dt -Z^H_t\,dW_t ,
\end{eqnarray}
where 
\begin{equation*}
\kappa^1_t = -A_t \frac{\phi^{(2)}(A_t)  }{\phi'(A_t)}, \qquad 
\kappa^2_t =   -A_t \frac{\phi^{(3)}(A_t)}{\phi^{(2)}(A_t)} .
\end{equation*}
Note that from \eqref{eq:lower_Y_estimate} and since $\phi' \leq 0$, $H_t \geq 0$ a.s. Hence the dynamics of $H$ is given by:
\begin{equation}\label{eq:dyn_BSDE_H}
- d H_t = F^H(t,H_t,Z^H_t) dt +  Z^H_t dW_t,
\end{equation}
where $F^H = F+ L$ with
\begin{eqnarray} \label{eq:gene_F}
F(t,h) &= &\frac{-1}{\eta_t \phi'(A_t)} \left[ f(\phi(A_t)-\phi'(A_t) h) - f(\phi(A_t)) +f'(\phi(A_t)) \phi'(A_t) h \right]  \mathbf 1_{h\geq 0}, \\ \label{eq:gene_L}
L(t,h,z) & = &- \frac{\lambda_t}{\phi'(A_t)} + \frac{\kappa^1_t}{2} \frac{(Z^A_t)^2}{A_t} + \frac{ \kappa^1_t \kappa^2_t}{2}  \left( \frac{Z^A_t}{A_t} \right)^2 h + \kappa^1_t \frac{Z^A_t}{A_t} z .
\end{eqnarray}
$L$ stands for the linear part of the generator of $H$, even if there is also a linear part in $F$:
$$h \geq 0 \mapsto \frac{-f'(\phi(A_t))}{\eta_t } h = \frac{\kappa^1_t }{\eta_t A_t}  h.$$ 
Note that the generator $F$ is {\it singular} in the sense of \cite{jean:reve:14}, since in general
$$\int_0^T \frac{\kappa^1_t }{\eta_t A_t}  dt =+\infty.$$
It is proved in \cite[Propositions 3.1 and 3.5]{jean:reve:14} that the terminal condition on $H$ should be zero. Hence we adopt their definition (\cite[Definition 2.1]{jean:reve:14}) of a solution. 
\begin{Def}[BSDE with singular generator] \label{def:sing_gene_sol}
We say that $(H,Z^H)$ solves the BSDE \eqref{eq:BSDE_H} if the relation \eqref{eq:dyn_BSDE_H} holds a.s. for any $t \in [0,T]$ and if 
$$\bE \left[  \int_0^T |F^H(s,H_s,Z^H_s)| ds +\left(  \int_0^T (Z^H_s)^2 ds \right)^{\frac{1}{2}} \right] < +\infty.$$
\end{Def}

We add some additional conditions on $f$:
\begin{enumerate}[label=\textbf{(C\arabic*)}]
\item \label{C1_concave} {\it $f$ is concave and of class $C^2$ on $(0,+\infty)$. }
\item \label{C2_tech} {\it There exist two constants $\delta > 0$ and $R > 0$ such that $x \mapsto \displaystyle G(x)^{-\delta} =  \left( \int_x^\infty \dfrac{1}{-f(y)} dy \right)^{-\delta}$ is convex on $[R,+\infty)$.}

\item \label{C3_tech} {\it If $\psi^\delta$ is the increasing and concave function $\psi^\delta : x \mapsto G^{-1} (x^{-1/\delta})$ for $x>0$, then $(-f)\circ \psi^\delta$ is also increasing and concave on a neighborhood of $+\infty$.}
\end{enumerate}
Let us emphasize that the conditions \ref{C2_tech} and \ref{C3_tech} only involve the function $f$ on some interval $[R,+\infty)$ and the value of $R$ may be large. Under \ref{C2_tech}, we know that $\psi^\delta$ is increasing. From \ref{C1_concave}, $-f'$ is a non-decreasing function and there exists a rank such that for any $x$ greater than this rank, $-f' > 0$. In other words $(-f)\circ\psi^\delta$ is an increasing function, at least on a neighborhood of $\infty$. Hence the main assumption in \ref{C3_tech} is the concavity of $-f \circ\psi^\delta$.

Finally our last condition on $f$ is the following. Let us define for some $\rho \in(0,1)$, the non-negative function $h(y)=-f(y) G(y)^{1-\rho}$.
\begin{enumerate}[label=\textbf{(C\arabic*)}]
\setcounter{enumi}{3}
\item \label{C4_tech} {\it There exists $\rho \in (0,1)$ such that the function $y \mapsto \dfrac{y}{h(y)}$ remains bounded on a neighborhood of $+\infty$.}
\end{enumerate}
Again this property only depends on the behavior of $f$ near $+\infty$. Note that 
$$\int_y^{+\infty} \frac{1}{h(t)}dt = \frac{1}{\rho} G(y)^\rho,$$
thus $1/h$ is integrable on $[1,+\infty)$. It is known (see \cite[Section 178]{hardy:08}) that if $h$ is non-decreasing, then we have
$$\lim_{y\to +\infty} \frac{y}{h(y)} = 0.$$
Thus \ref{C4_tech} holds. But since 
$$h'(y) = G(y)^{-\rho} \left[-f'(y) G(y) - (1-\rho)\right] = G(y)^{-\rho} \left[ \frac{G''(y) G(y) +(\rho-1) (G'(y))^2}{(G'(y))^2} \right],$$
the non-negativity of $h'$ is equivalent to the non-negativity of the second derivative of $G^\rho$. In other words $h$ is non decreasing if and only if $G^\rho$ is convex. 

To state our main result, we assume some regularity on the process $\gamma:=\dfrac{1}{\eta}$. Let us suppose that $\gamma$ is an It\^o process of the form:
\begin{equation} \label{eq:gamma_ito_proc}
d\gamma_t=b^\gamma_t\,dt+\sigma^\gamma_t\,dW_t.
\end{equation}
Note that with Condition \ref{A1}, it is equivalent to assume that $\eta$ is an It\^o process. We also consider the stronger case:
\begin{equation} \label{eq:gamma_SDE}
d\gamma_t=b(\gamma_t) \,dt+\sigma(\gamma_t) \,dW_t.
\end{equation}
In the rest of the paper we call \ref{H_diff} the next condition:
\begin{enumerate}[label=\textbf{(H)}]
\item \label{H_diff} {\it The process $1/\eta$ satisfies 
\begin{itemize}
\item either the equation \eqref{eq:gamma_ito_proc} with bounded coefficients $b^\gamma$ and $\sigma^\gamma$;
\item or the SDE \eqref{eq:gamma_SDE} where the functions $b$ and $\sigma$ are Lipschitz continuous and $\sigma$ is bounded.
\end{itemize}
 }
\end{enumerate}
Let us now state our main result.
\begin{Thm} \label{thm:uniq_sing_BSDE}
Under the hypotheses {\rm \ref{A1}} to {\rm \ref{A3}}, {\rm \ref{C1_concave}} to {\rm \ref{C4_tech}} and {\rm \ref{H_diff}}, the BSDE \eqref{eq:sing_BSDE} has a unique solution $(Y,Z^Y)$ (in the sense of Definition \ref{def:sing_cond_sol}). This solution is given by \eqref{eq:asymp_beha_Y}: a.s. $Y_t = \phi(A_t) - \phi'(A_t) H_t$ for any $t\in [0,T]$, where $H$ is the unique solution of the BSDE \eqref{eq:BSDE_H}  (in the sense of Definition \ref{def:sing_gene_sol}).
\end{Thm}

\subsection{Examples} \label{ssect:examples}

Let us study several functions $f$, ordered by their ``non linearity'', and verify the conditions of the previous theorem. 

 If $f(y)=-(y+1)|\log(y+1)|^q$ for some $q>1$ and $y\geq 0$, all conditions {\rm \ref{A1}}, {\rm \ref{A2}} and {\rm \ref{A3}} are verified and if $p$ is the H\"older conjugate of $q$ then 
$$\forall x > 0, \quad G(x) =\dfrac{1}{q-1} \log(x+1)^{1-q},\quad \phi(x) = \exp\left( \left((q-1)x\right)^{1-p}\right)-1.$$
Direct computations show that $ - \dfrac{\phi''(x)}{\phi'(x)} x =  \left[ (p-1)^p  x^{-p} + p\right]$
is not a bounded function near zero and for any $\delta > 0$, $G^{-\delta}$ is not convex. 
Somehow this function $f$ is ``not enough non linear''. In other words, for this class of functions, we only have existence of a minimal solution (Proposition \ref{prop:exist_min_sol}). 

All of the next functions verify {\rm \ref{C1_concave}} to {\rm \ref{C3_tech}}. 
\begin{itemize}
\item {\it Power case: } $f(y)=-y|y|^q$ for some $q > 0$. Then $G(x) = \dfrac{1}{qx^q}$ and $\phi(x) = \left(\dfrac{1}{qx}\right)^{\frac{1}{q}}$. The assumption {\rm \ref{C2_tech}} holds for any $\delta > 0$ and $$ - \frac{\phi''(x)}{\phi'(x)} x = \frac{q+1}{q},\quad - \dfrac{\phi^{(3)}(x)}{\phi^{(2)}(x)} x = \dfrac{1+2q}{q}.$$

\item {\it Exponential case:}  $f(y) = -(\exp(ay)-1)$ for some $a> 0$. Then 
$$G(x) = - \frac{1}{a} \log \left(1-e^{-ax}\right) = G^{-1}(x)=\phi(x).$$
And
$$ - \frac{\phi^{(2)}(x)}{\phi'(x)} x = \frac{a x e^{ax}}{e^{ax}-1}\underset{x\to 0}{\sim} 1, \quad -\frac{\phi^{(3)}(x)}{\phi^{(2)}(x)}x = \frac{a x (1+e^{ax})}{e^{ax}-1} \underset{x\to 0}{\sim} 2,$$
are bounded near zero. From the Lemmata \ref{lem:bounded_cond} and \ref{lem:bound_coef_kappa_2}, \ref{C2_tech} and \ref{C3_tech} hold. 

. 

\item {\it Normal case:} $f(y)= -\exp(a y^2)$ for some $a> 0$ and $y\geq 0$ (note that $f(0)=-1$ to simplify the computations). Then 
$$G(x) = \sqrt{\frac{\pi}{ a}} \left[1-\cN (x\sqrt{2a})\right],\quad \phi(x)=\dfrac{1}{\sqrt{2a}}\cN^{-1}\left( 1-x \sqrt{\dfrac{a}{\pi }}\right),$$
where $\cN(\cdot)$ is the cumulative distribution function of the normal law. 
Thereby 
$$\phi'(x) = -\exp(a\phi(x)^2), \quad \phi^{(2)}(x) = 2a\phi(x) (\phi'(x))^2,$$
and with $\sqrt{a} \phi(x)=z/\sqrt{2}$ 
\begin{eqnarray*}
- \frac{\phi''(x)}{\phi'(x)} x &=& 
z\sqrt{2\pi} \exp(z^2/2) \left[1-\cN (z)\right] \underset{z\to +\infty}{\sim} 1.
\end{eqnarray*}
using the classical tail estimate of the normal law. Hence $x\phi''(x)/\phi'(x)$ is bounded near zero and {\rm \ref{C2_tech}} holds (again by Lemma \ref{lem:bounded_cond}).

Since
$$ \phi^{(3)}(x)=2a (\phi'(x)^3) + 4a \phi(x) \phi'(x) \phi^{(2)}(x) ,$$
Thus
$$ -\frac{\phi^{(3)}(x)}{\phi^{(2)}(x)}x = -x\frac{\phi'(x)}{\phi(x)} - 2x \frac{\phi^{(2)}(x)}{\phi'(x)}.$$
From \ref{C2_tech} and Lemma \ref{lem:bounded_cond}, the second term is bounded. Arguing as above yields to:
$$-x\frac{\phi'(x)}{\phi(x)}  = \dfrac{\sqrt{2a}}{z} G  \left( \dfrac{z}{\sqrt{2a}} \right) e^{-z^2/2} \underset{z\to +\infty}{\longrightarrow} 0.$$
From Lemma \ref{lem:bound_coef_kappa_2}, \ref{C3_tech} is verified.  
\end{itemize}
For all function direct computations show that $\lim_{y\to +\infty} \dfrac{y}{h(y)} =0$, for any $\rho \in (0,1)$. Hence \ref{C4_tech} holds.

\subsection{Properties of the generator $F^H$} \label{ssect:generator_f_H}

Let us describe the properties of the generator $F^H$ given by \eqref{eq:gene_F} and  \eqref{eq:gene_L}. 
\begin{itemize}
\item $F^H(t,h,z)=L(t,h,z)$, for any $h\leq 0$.
\item $F$ is continuous and monotone w.r.t. $h$: for any $h$ and $h'$,
$$(h-h')(F(t,h)-F(t,h')) \leq \frac{\kappa^1_t }{\eta_t A_t} (h-h')^2,$$
since $f$ is itself monotone.
\item For any $|h| \leq r$, 
$$|F(t,h)-F(t,0)| \leq \frac{\kappa^1_t }{\eta_t A_t} r - \frac{f(\phi(A_t)+\phi'(A_t) r)}{\phi'(A_t) \eta_t}.$$ 
\item $L$ is linear w.r.t. $h$ and $z$ with
$$|L(t,h,z)-L(t,h',z')| \leq \frac{ \kappa^1_t \kappa^2_t}{2}  \left( \frac{Z^A_t}{A_t} \right)^2 |h-h'| +  \left| \kappa^1_t \frac{Z^A_t}{A_t} \right| |z-z'|.$$
\item The process $F^H(\cdot,0,0)$ is equal to
$$F^H(t,0,0)= L(t,0,0)=\frac{\kappa^1_t}{2} \frac{(Z^A_t)^2}{A_t} - \frac{\lambda_t}{\phi'(A_t)}.$$
\end{itemize}
Therefore, to deduce existence or uniqueness of the solution of \eqref{eq:BSDE_H}, we cannot directly use some results in \cite{jean:reve:14} or in \cite{pard:rasc:14}. Indeed if we consider this BSDE \eqref{eq:BSDE_H} with $f\equiv 0$, we get a linear BSDE. From our best knowledge, existence of a solution is proved only under some exponential moment condition on the coefficients (see \cite[Proposition 5.31]{pard:rasc:14}). Even if we avoid the final time $T$, then $1/A$ is bounded on $[0,T-\eps]$ (Inequality \eqref{eq:bound_A}), but $Z^A$ is only BMO. Hence the stochastic exponential of the martingale $M=\int_0^\cdot Z^A_s dW_s$ is uniformly integrable. But controlling the exponential of the bracket of $M$ is more difficult. If $\xi = \int_0^T 1/\eta_s ds$ is Malliavin differentiable, then we require that its Malliavin derivative has exponential moments.

Let us emphasize if we do not control the quantity $(Z^A)^2/A^2$, then from \cite[Proposition 3.1]{jean:reve:14}, we may have infinitely many solutions.

\subsubsection{On the coefficients $\kappa^1$ and $\kappa^2$}

To obtain some estimates on $\kappa^i$, $i=1,2$, which depend essentially on $\phi$ and its derivatives, we need additional conditions on $f$. Recall that $\phi$ is a positive function and it solves, $\phi' = f\circ \phi$. Thus $\phi' \leq 0$ and $\phi^{(2)} = (f'\circ \phi) \phi' \geq 0$ from Assumption \ref{A2}. Note that the functions $\phi$ and $\phi'$ are continuous and bounded on $[\epsilon,+\infty)$ for any $\epsilon> 0$ and $\phi'$ never reaches zero on compact subset on $(0,\infty)$. Hence the processes $\kappa^1$ and $F^H(\cdot,0,0)$ are bounded on any time interval $[0,T-\epsilon]$ for $0< \epsilon <T$ and are nonnegative. Moreover since $\phi(0) = +\infty$, the condition \ref{A3} implies that $\phi'(0) = f(+\infty)=-\infty$. Using \eqref{eq:bound_A} and \ref{A1}, we deduce that $\lambda/\phi'(A)$ is bounded on $[0,T]$ by $ \|\lambda\|(-\phi'(T/\etamin))^{-1}$.

The next lemma provides a necessary and sufficient condition for the boundedness of the process $\kappa^1$.  
\begin{Lem} \label{lem:bounded_cond}
The next two assertions are equivalent.
\begin{enumerate}
\item There exists a constant $\delta > 0$ and $R > 0$ such that $x \mapsto G(x)^{-\delta}$ is convex on $[R,+\infty)$ (assumption {\rm \ref{C2_tech}}). 
\item The function $\phi$ verifies the next property: there exists $K > 1$ and $\varrho > 0$ such that for all $x \in (0,\varrho ]$: $\left| x \dfrac{\phi''(x)}{\phi'(x)}  \right| \leq K.$
\end{enumerate}
The constants are related by: $\varrho = 1/R$ and $ \delta = K-1$.
\end{Lem}
\begin{proof}
Remark that 
$ x \dfrac{\phi''(x)}{\phi'(x)}  \leq 0$. Hence it is enough to show that there exists $K > 0$ such that $- x \dfrac{\phi''(x)}{\phi'(x)} \leq K$. W.l.o.g. we can assume that $K > 1$. 
Now let us define $\varphi$ by $\varphi(x)=\phi(1/x)$ for any $x>0$. Then 
$$ \phi'(x) =-\frac{1}{x^2}\varphi'(1/x) , \qquad \phi''(x) = \frac{2}{x^3} \varphi'(1/x) + \frac{1}{x^4} \varphi''(1/x).$$
Thus
$$-x \dfrac{\phi''(x)}{\phi'(x)}  =  2+ \frac{1}{x} \frac{\varphi''(1/x)}{\varphi'(1/x)}.$$
Hence to establish Lemma \ref{lem:bounded_cond} it is sufficient to prove that there exists $K>1$ and $\varrho >0$ such that for all $t\geq 1/\varrho =R$,
\[
- t \frac{\varphi''(t)}{\varphi'(t)}\geq 2-K = -(K-2).
\]
Let us rewrite this condition in terms of the so-called Arrow-Pratt coefficient of absolute risk aversion by interpreting $\varphi$ as utility function,
\begin{equation} \label{pratt-condition}
\alpha_\varphi(t):=-\frac{\varphi''(t)}{\varphi'(t)}\geq -\frac{(K-2)}{t}=:\alpha_{K-2}(t),
\end{equation}
where the utility function to $\alpha_K$ is given (up to positive affine transformations) by $u_K(t) = t^{K-1}$. 
By a classical theorem due to Pratt (\cite{prat:64}, see also \cite[Proposition 2.44]{foel:schi:11}), Condition \eqref{pratt-condition} holds if and only if
\begin{equation} \label{pratt-transformation}
\varphi=\psi \circ u_K
\end{equation}
for a strictly increasing concave function $\psi$. As $\varphi=G^{-1}(1/\cdot)$, Pratt's condition is equivalent to
\[
\psi(t):=G^{-1}\left(\frac{1}{u_K^{-1}(t)}\right) = G^{-1} \left( t^{-\frac{1}{K-1}} \right).
\]
defines a strictly increasing concave function. In other words $\displaystyle x \mapsto G(x)^{1-K}$ is strictly increasing and convex.
This achieves the proof of the Lemma.
%
\end{proof}

Under the condition \ref{C2_tech}, 
using \eqref{eq:bound_A}, 
$A_t \leq 1/R$, if $T- \frac{\etamin}{R} \leq t \leq T $. Thus the process $\kappa^1 = \left[  -\dfrac{\phi^{(2)}(A)}{\phi'(A)} A \right]$ is bounded on this interval: 
On the rest of the interval $[0,T]$ this coefficient is also bounded due to the regularity of $f$ (Conditions \ref{A1} and \ref{A2}).

%
%


\bigskip 

From \ref{C1_concave}, $\phi^{(3)} = (f^{(2)} \circ \phi) (\phi')^2 + (f' \circ \phi) \phi^{(2)} \leq 0$ and we deduce that $\kappa^2$ is also nonnegative. 

\begin{Lem} \label{lem:bound_coef_kappa_2}
Assume that {\rm \ref{C1_concave}}, {\rm \ref{C2_tech}} and {\rm \ref{C3_tech}} hold. Then the function $x\mapsto  -x\dfrac{\phi^{(3)}(x)}{\phi^{(2)}(x)}$ is a non-negative and bounded function on a neighborhood of zero.
\end{Lem}
\begin{proof}
Remark that  $x\phi^{(3)}(x) /\phi^{(2)}(x) \leq 0$. The conclusion of the lemma is equivalent to the boundedness from below of $-x \widehat \phi^{(2)} (x)/ \widehat \phi'(x)$ in the neighborhood of $\infty$ with $\widehat \phi(x) = -\phi'(1/x)$. From the proof of Lemma \ref{lem:bounded_cond}, we have 
$$\widehat \phi(x) = (-f)(\psi^\delta \circ u_K) = \widehat{\psi^\delta} \circ u_{K}$$
where $\psi^\delta$ and $\widehat{ \psi^\delta}$ are increasing and concave. Note that we can assume w.l.o.g. that the constant $K$ is the same. Indeed if \ref{C2_tech} holds for some $\delta > 0$, the same condition holds for any $\delta' \geq \delta$. Hence boundedness is equivalent to the existence of $K > 1$ such that $x \mapsto G^{-1} (x^{-1/(K-1)})$ (condition \ref{C2_tech}) and $(-f)\circ \psi^\delta$ are increasing and concave.
\end{proof}
\begin{Rem} 
Note that under {\rm \ref{C1_concave}}, the boundedness of $\kappa_2$ is equivalent to condition {\rm \ref{C3_tech}}.
\end{Rem}
Under \ref{C1_concave} to \ref{C3_tech}, the processes $\kappa^1$ and $\kappa^2$ are non-negative and bounded.

\subsubsection{Properties of $Z^A$}

In the linear part $L$ of the generator $F^H$ given by \eqref{eq:gene_L}, we also have to control the process $Z^A$. First note that the martingale $\left( \int_0^t Z^A_s dW_s, \ t \in [0,T]\right)$ is a BMO martingale (see \cite{kaza:94}) and $Z^A \in \bH^q((0,T))$, $q>1$, due to the assumption that $\eta$ is bounded from above and away from zero.

\begin{Lem}
For any $\rho \in (0,1)$ and $p>1$, we have
\begin{equation} \label{eq:estim_ZA}
\bE \left[ \left(  \int_0^T  \frac{(Z^A_s)^2}{(A_s)^{1+\rho}}  ds \right)^p  \right] < +\infty.
\end{equation}
\end{Lem}
\begin{proof}
Let us apply It\^o's formula to $A^{1-\rho}$ on $[0,T-\eps]$: 
\begin{eqnarray*}
A_t^{1-\rho}  &= & (A_{T-\eps})^{1-\rho} + \int_t^{T-\eps} (1-\rho) \frac{(A_s)^{-\rho} }{\eta_s} ds -\frac{1}{2} \int_t^{T-\eps} (1-\rho)(-\rho) \frac{(Z^A_s)^2}{(A_s)^{1+\rho}}  ds \\
&+& \int_t^{T-\eps} (1-\rho) (A_s)^{-\rho}  Z^A_s dW_s.
\end{eqnarray*}
Hence,
\begin{eqnarray} \label{eq:estim_Z_A_1}
\frac{(1-\rho)\rho}{2} \int_0^{T-\eps} \frac{(Z^A_s)^2}{(A_s)^{1+\rho}} ds & = & A_0^{1-\rho}  -(A_{T-\eps})^{1-\rho} - (1-\rho) \int_0^{T-\eps} \frac{1}{\eta_s (A_s)^\rho} ds\\ \nonumber
& - & (1-\rho) \int_0^{T-\eps} (A_s)^{-\rho}  Z^A_s  dW_s.
\end{eqnarray}
Taking the expectation and using \eqref{eq:bound_A} and the fact that $t\mapsto (T-t)^{-\rho}$ is integrable at time $T$, we can apply Lebesgue monotone convergence theorem to get
\begin{equation*} 
\bE \int_0^{T} \frac{(Z^A_s)^2}{(A_s)^{1+\rho}} ds =  \frac{2}{\rho} \bE \left( \frac{(A_0)^{1-\rho}}{1-\rho} -  \int_0^{T} \frac{1 }{\eta_s (A_s)^\rho} ds  \right) < +\infty.
\end{equation*}
Using \eqref{eq:estim_Z_A_1} for any $p>1$ we obtain for some constant $C_p>0$,
\begin{eqnarray*}
\frac{1}{C_p}\left|\int_0^{T-\eps}\frac{(Z_s^A)^2}{(A_s)^{1+\rho}}\,ds\right|^p&\leq& \ |(A_{T-\varepsilon})^{1-\rho}|^p+|(A_{0})^{1-\rho}|^p+\left|\int_0^{T-\eps} \frac{1 }{\eta_s(A_s)^\rho}\,ds\right|^p\\
& +& \left|\int_0^{T-\eps} (A_s)^{-\rho} Z^A_s dW_s\right|^p \\
& \leq & |(A_{0})^{1-\rho}|^p +  |(T/\etamin )^{1-\rho}|^p+\frac{1}{\etamin^p} \left( \int_0^{T} \frac{1 }{(A_s)^\rho}\,ds\right)^p \\
& +& \left|\int_0^{T-\eps} (A_s)^{-\rho} Z^A_s dW_s\right|^p.
\end{eqnarray*}
From the BDG and H\"older inequalities, taking the expectation leads to
\begin{eqnarray*}
\frac{1}{C_p}\bE \left|\int_0^{T-\eps}\frac{(Z_s^A)^2}{(A_s)^{1+\rho}}\,ds\right|^p & \leq & \bE \left[ |(A_{0})^{1-\rho}|^p +  |(T/\etamin )^{1-\rho}|^p+\frac{1}{\etamin^p} \left( \int_0^{T} \frac{1 }{(A_s)^\rho}\,ds\right)^p \right]\\
& +& \left\{ \bE \left[ \left|\int_0^{T-\eps} ((A_s)^{-\rho} Z^A_s)^2 ds\right|^{p}\right] \right\}^{1/2}.
\end{eqnarray*}
The function $x\mapsto x^{1-\rho}$ is bounded on $[0,T/\etamin]$ by some constant $C$. Thus
\begin{eqnarray*}
\frac{1}{C_p}\bE \left|\int_0^{T-\eps}\frac{(Z_s^A)^2}{(A_s)^{1+\rho}}\,ds\right|^p & \leq & \bE \left[|(A_{0})^{1-\rho}|^p +  |(T/\etamin )^{1-\rho}|^p+\frac{1}{\etamin^p} \left( \int_0^{T} \frac{1 }{(A_s)^\rho}\,ds\right)^p \right]\\
& +& C^{p/2} \left\{ \bE \left[ \left|\int_0^{T-\eps} \frac{(Z^A_s)^2}{(A_s)^{1+\rho}} ds\right|^{p}\right] \right\}^{1/2}.
\end{eqnarray*}
In other words if $\gamma_\eps = \bE \left|\int_0^{T-\eps}\frac{(Z_s^A)^2}{(A_s)^{1+\rho}}\,ds\right|^p$, then there exists $C$ independent of $\eps$ such that 
$$0\leq \gamma_\eps \leq C  (1+(\gamma_\eps)^{1/2}),$$
which leads to the existence of some constant $C$ such that $\gamma_\eps \leq C$. Using the monotone convergence theorem, we obtain the desired estimate. 
\end{proof}

From this estimate on $Z^A$, using \eqref{eq:bound_A}, we have for any $p>1$
\begin{equation} \label{eq:estim_ZA_bis}
\bE \left[ \left(  \int_0^{T} \frac{(Z^A_s)^2}{A_s} ds \right)^p \right]  \leq \left(\dfrac{T}{\etamin}\right)^{\rho p} \bE \left[ \left(  \int_0^{T} \frac{(Z^A_s)^2}{(A_s)^{1+\rho}} ds \right)^p \right] <+\infty.
\end{equation}
Thereby we deduce the next corollary.
\begin{Coro} \label{lem:integrability_f_H_0}
Under Assumptions {\rm \ref{A1}} and  {\rm \ref{C2_tech}}, the process $F^H(\cdot,0,0)$ belongs to $\bL^p(\Omega \times [0,T])$ for any $p>1$.
\end{Coro}
\begin{proof}
Since
$$F^H(t,0,0)= \frac{\kappa^1_t}{2} \frac{(Z^A_t)^2}{A_t} + \frac{\lambda_t}{\phi'(A_t)},$$
immediate consequence of Inequality \eqref{eq:estim_ZA_bis} and boundedness of the coefficients $\lambda$ and $\kappa^1$ due to \ref{A1} and \ref{C2_tech}.
\end{proof}

Recall that the generator $F^H$ depends on $z$ linearly with a coefficient equal to $\kappa^1 \dfrac{Z^A}{A}$.  We can control this term if we can apply Girsanov's theorem, that is if the stochastic exponential of the martingale $\left( t \mapsto  \displaystyle \int_0^t \dfrac{Z^A_s}{A_s} dW_s,\ t\in[0,T]\right)$ is uniformly integrable. Moreover in $L$, we also have a term equal to $\kappa^1_t \kappa^2_t \left( \dfrac{Z^A_t}{A_t} \right)^2 h$, which requires to control the quadratic variation of the previous martingale. In general it seems to be difficult to handle these terms. However under Assumption \ref{H_diff}, we can overcome this problem. 

Suppose that the process $1/\eta$ is given by \eqref{eq:gamma_ito_proc} and remark that:
\begin{eqnarray} \nonumber
A_t & = & \bE \left[ \int_t^T \frac{1}{\eta_s} ds \bigg| \cF_t \right] = \bE \left[ \int_t^T\gamma_s ds \bigg| \cF_t \right]  \\ \nonumber
& = & \bE \left[ \int_t^T \left( \gamma_t + \int_t^s b^\gamma_u du + \int_t^s \sigma^\gamma_u dW_u\right) ds \bigg| \cF_t \right] \\ \label{eq:decomp_A_Ito_case} 
&= &\gamma_t (T-t) + \bE \left[  \int_t^T (T-u) b^\gamma_u du  \bigg| \cF_t \right].
\end{eqnarray}
If we denote 
$$\widetilde A_t = \bE \left[  \int_t^T (T-u) b^\gamma_u du  \bigg| \cF_t \right],$$
then 
$$Z^A_t = \sigma^\gamma_t (T-t) + Z^{\widetilde A}_t,\quad \mbox{with} \quad -d \widetilde A_t = (T-t) b^\gamma_t dt + Z^{\widetilde A}_t dW_t .$$

\begin{Lem} \label{lem:expo_moment_1}
If Condition {\rm \ref{H_diff}} holds, there exist a deterministic time $\widehat T < T$ and a positive constant $\Sigma > \frac{1}{2} \vee \|\kappa^1\|_\infty^2$ such that 
\begin{equation} \label{eq:exponen_moment_Z^A}
\bE \left[ \exp \left( \Sigma \int_{\widehat T}^T  \left( \frac{Z^A_s}{A_s} \right)^2 ds \right) \right] < +\infty.
\end{equation}
\end{Lem}
\begin{proof}
Since $b^\gamma$ and $\sigma^\gamma$ are essentially bounded in \eqref{eq:gamma_ito_proc} and since $Z^A_t = \sigma^\gamma_t (T-t) + Z^{\widetilde A}_t$, using \eqref{eq:bound_A}, it is sufficient to prove the desired estimate only for $Z^{\widetilde A}$. 
Since $b^\gamma$ is essentially bounded, then 
$$ \left| \widetilde A_t \right| \leq \bE \left[  \int_t^T (T-u)  \left\| b^\gamma \right\|du  \bigg| \cF_t \right]=\frac{\|b^\gamma\|}{2}(T-t)^2.$$
It\^o's formula leads to
\begin{eqnarray*}
-d\left(  \frac{\widetilde A_t}{(T-t)} \right) & = & - \frac{\widetilde A_t}{(T-t)^{2}} dt +  b^\gamma_t dt + \frac{Z^{\widetilde A}_t}{(T-t)} dW_t.
\end{eqnarray*}
Hence we obtain that the martingale $\left( \displaystyle \widetilde M_u =  \int_0^u \dfrac{Z^{\widetilde A}_t}{(T-t)} dW_t, \ u \in [0,T]\right)$ is a BMO martingale:
$$\forall u \in [0,T], \ | \widetilde M_T-  \widetilde M_u| \leq 2 \| b^\gamma\| (T-u) \Rightarrow \sup_{u\in [t,T]} \bE \left[ | \widetilde M_T-  \widetilde M_u|  \bigg| \cF_u \right] \leq 2 \| b^\gamma\| (T-t).$$
Therefore we can choose $\widehat T$ very close to $T$ such that the BMO norm of $ \widetilde M$ on $[\widehat T,T]$ is as small as required. Using the Nirenberg inequality (see \cite[Theorem 2.2]{kaza:94}), there exists a constant $\widetilde \Sigma$ depending on the BMO norm of $\widetilde M$, such that 
$$\bE \left[  \exp\left( \widetilde \Sigma \int_{\widehat T}^T  \left( \frac{Z^{\widetilde A}_s}{A_s} \right)^2ds \right)  \right] < +\infty.$$
Precisely $\widetilde \Sigma$ should be smaller than the inverse of the BMO norm of $\widetilde M$. 
Thereby choosing $\widehat T$ sufficiently close to $T$, we get Condition \eqref{eq:exponen_moment_Z^A} for $Z^{\widetilde A}$.


Let us now assume that the process $\gamma=1/\eta$ solves the SDE \eqref{eq:gamma_SDE}, where $b$ and $\sigma$ are Lipschitz continuous functions defined on $\bR$. From \cite[Theorems 2.2.1 and 2.2.2]{nual:06}, the coordinate $\gamma_t$ belongs to $\bD^{1,\infty}$ for any $t\in [0,T]$. Moreover for any $p\geq 1$
\begin{equation} \label{eq:estim_deriv_xi_Lp}
\sup_{0\leq r\leq T} \bE \left( \sup_{r \leq t \leq T} |D_r\gamma_t |^p \right) < +\infty.
\end{equation}
The derivative $D_r\gamma_t$ satisfies the following linear equation:
\begin{eqnarray*}
D_r\gamma_t & = & \sigma(\gamma_r) +\int_r^t \widetilde \sigma(s) D_r\gamma_s dW_s +  \int_r^t \widetilde b(s) D_r \gamma_s ds
\end{eqnarray*}
for $r \leq t$ a.e. and $D_r \gamma_t = 0$ for $r>t$ a.e., where $\widetilde b(s)$ and $\widetilde \sigma(s)$ are two bounded processes, such that if $\mfd$ and $\sigma$ are of class $C^1$, they are given by:
$$
\widetilde b(s) = (\partial_{x} b) (\gamma_s),\qquad \widetilde \sigma(s) = (\partial_{x}  \sigma)  (\gamma_s).
$$
In this case $\xi = \int_0^T (T-u) b^\gamma_u  du$ is in $\bD^{1,2}$ and by the Clark-Ocone formula, 
we have
$$Z^{\widetilde A}_t = \int_t^T (T-u) \bE \left[ D_t b^\gamma_u  \bigg| \cF_t \right] du = \int_t^T (T-u) \bE \left[ \widetilde b(u) D_t \gamma_u  \bigg| \cF_t \right] du.$$
Since $\zeta_u = D_t \gamma_u$ satisfies the linear one-dimensional SDE:
$$\zeta_u = \sigma(\gamma_t) + \int_t^u \widetilde \sigma(s) \zeta_s dB_s + \int_t^u \widetilde b(s) \zeta_s ds,$$
an explicit formula for $|\zeta_u|$ reads
\begin{eqnarray*}
|\zeta_u| & = & |\sigma(\gamma_t)| \exp \left[ \int_t^u \widetilde \sigma(s)dB_s  - \frac{1}{2} \int_t^u \widetilde \sigma(s)^2 ds \right] \exp \left[ \int_t^u \widetilde b(s) ds \right].
\end{eqnarray*}
Since the function $b$ is supposed to be Lipschitz continuous, $\widetilde b$ is essentially bounded. Thereby 
\begin{eqnarray*}
\left| Z^{\widetilde A}_t \right| & \leq & \|\widetilde b\|  \int_t^T(T-u) \bE \left[  \left|  D_t \gamma_u \right|  \bigg| \cF_t \right] du \\
& = &   \|\widetilde b\|  |\sigma(\gamma_t)| \int_t^T(T-u) \bE \left[  \exp \left[ \int_t^u \widetilde b(s) ds \right]  \bigg| \cF_t \right] du\\
&\leq &\frac{1}{2}  \|\widetilde b\| e^{T \|\widetilde b\| }  |\sigma(\gamma_t)| (T-t)^2 = C (T-t)^2 |\sigma^\gamma_t|.
\end{eqnarray*}
Finally it implies that 
$$\frac{Z^A_t}{T-t} = \sigma^\gamma_t \left(1 + (T-t) \varsigma^\gamma_t \right),$$
where $\varsigma^\gamma$ is a bounded process. Hence Condition \eqref{eq:exponen_moment_Z^A} holds. 
\end{proof} 

\subsection{Proof of Theorem \ref{thm:uniq_sing_BSDE}} \label{ssect:proof_main_thm}

We suppose that the assumptions of Theorem \ref{thm:uniq_sing_BSDE} hold. First we prove the next result:
\begin{Lem} \label{lem:control_vartheta}
For any $\eps > 0$, there exists a deterministic time $T^\eps \in [0,T)$ such that a.s. for any $t\in [T^\eps,T]$
\begin{equation} \label{eq:a_priori_estim_Y_bis}
\vartheta(T-t)\leq \phi \left( \frac{T-t}{(1+\eps)\etamax}\right) = \phius_\eps(t).
\end{equation}

\end{Lem}
\begin{proof}
Recall that $\vartheta = \widetilde G^{-1}$ solves \eqref{eq:upper_ODE} with 
$$\widetilde G(x) = \int_x^\infty \frac{-1}{ \lambdamax + \frac{f(y)}{\etamax}}  dy = \etamax  \int_x^\infty \frac{1}{- \lambdamax \etamax -f(y)} dy,$$
and with $\vartheta(0)=+\infty$. But 
\begin{eqnarray*}
\widetilde G(x) & = & \etamax  \int_x^\infty \frac{1}{-\mathfrak C -f(y)} dy = \etamax G(x) +  \lambdamax (\etamax )^2\int_x^\infty \frac{1}{(-\lambdamax \etamax -f(y))(-f(y))} dy.
\end{eqnarray*}
Therefore $\etamax G(x)  \leq \widetilde G(x)$ and 
\begin{eqnarray*}
\widetilde G(x)  &\leq &  \etamax G(x) +   \lambdamax (\etamax )^2 \frac{G(x)}{(-  \lambdamax \etamax  -f(x))}  = \etamax G(x) \left( \frac{-f(x)}{-\lambdamax \etamax -f(x)}\right) .
\end{eqnarray*}
We deduce that on the interval $\widetilde I = (R,+\infty)$, with $R=f^{-1}(-(1+\eps)\lambdamax \etamax)$, $\etamax G(x)  \leq\widetilde G(x) \leq (1+\eps) \etamax G(x)$, and thereby on the neighborhood of zero $(0,G^{-1}(R / \etamax))$,
$$G^{-1} (x/\etamax) \leq  \vartheta(x) \leq G^{-1} (x/((1+\eps)\etamax))= \phi(x/((1+\eps)\etamax)).$$
Thus provided that $T^\eps=T-G^{-1}(R / \etamax)\leq t \leq T$, $ \vartheta(T-t) \leq\phius_\eps(t)$.
\end{proof}
%
%

\begin{Lem} \label{lem:integrable_0_T}
The non-negative process $\left(\dfrac{-f'(\phi(A_s))}{\eta_s} H_s =  \dfrac{\kappa^1_s}{\eta_s A_s} H_s, \ s \in [0,T)\right)$ is integrable on $\Omega \times [0,T]$. 
\end{Lem}
\begin{proof}
From the a priori estimate \eqref{eq:a_priori_estim_Y} on $Y$, we obtain that $0 \leq H_t \leq \dfrac{\vartheta(T-t)}{-\phi'(A_t)}.$
Indeed using the very definition of $\kappa^1_s$ we have
\begin{eqnarray*}
0\leq \dfrac{\kappa^1_s}{\eta_s A_s} H_s & = &  \frac{1}{\eta_s}\frac{\phi^{(2)}(A_s)}{(-\phi'(A_s))}  H_s \leq  \frac{1}{\etamin}\frac{\phi^{(2)}(A_s)\vartheta \left(T-s \right)}{(\phi'(A_s))^2} .
\end{eqnarray*}
Since $\phi^{(2)}$ is a non-increasing function, from the inequality \eqref{eq:bound_A}, for any $\eps > 0$
$$\phi^{(2)}(A_s) \leq \phi^{(2)}\left( \frac{T-s}{\etamax} \right) \leq \phi^{(2)}\left( \frac{T-s}{(1+\eps)\etamax} \right).$$
Moreover from Lemma \ref{lem:bounded_cond} we know that for any $y \in (0,1/R)$ and $0< a\leq 1$, 
\begin{equation} \label{eq:increment_deriv_phi}
1\leq \frac{-\phi'(ay)}{-\phi'(y)} \leq \frac{1}{a^K}.
\end{equation}
Since $\phi'$ is non-decreasing, with \eqref{eq:bound_A} and $a = \etamin/((1+\eps)\etamax)$,
$$-\phi'(A_s) \geq -\phi'\left( \dfrac{T-s}{\etamin} \right) \geq -a^K  \phi'\left( \dfrac{T-s}{(1+\eps)\etamax} \right).$$
Therefore we obtain that on some deterministic neighborhood of $T$:
$$\frac{\phi^{(2)}(A_s)\vartheta \left(T-s \right)}{(\phi'(A_s))^2} \leq a^{-K}  \left( \frac{\phi^{(2)} \phi}{(\phi')^2} \right) \left( \frac{T-s}{(1+\eps)\etamax} \right).$$
Hence we have a deterministic bound on $ \dfrac{\kappa^1}{\eta A} H$. 
Now 
$$\int_t^u  \frac{1}{\etamax} \left( \frac{\phi''\phi}{(\phi')^2} \right) \left( \frac{T-s}{\etamax} \right) ds = \int_{(T-u)/(\etamax)}^{(T-t)/(\etamax)} \left( \frac{\phi''\phi}{(\phi')^2} \right) \left(x \right) dx$$
and our result follows if $(\phi''\phi)/(\phi')^2$ is integrable at zero. Remark that 
$$ \left( \frac{\phi}{\phi'} \right)^\prime = 1- \frac{\phi''\phi}{(\phi')^2} .$$
Hence the integrability is equivalent to the existence of the limit at zero of $\phi/\phi'$, that is the limit at infinity of $y\mapsto y /(-f(y))$. From \cite[Section 178]{hardy:08}, since $-f$ is non-decreasing, this limit is equal to zero. This achieves the proof of the lemma.
\end{proof}

We show that the BSDE \eqref{eq:BSDE_H} has a solution on $[0,T]$.
\begin{Prop} \label{prop:existence_sol_BSDE_sing_gen}
Under Conditions {\rm \ref{A1}} to {\rm \ref{A3}} and {\rm \ref{C1_concave}} to {\rm \ref{C4_tech}}, there exists a non-negative process $(H,Z^H)$ solution (in the sense of Definition \ref{def:sing_gene_sol}) of the BSDE \eqref{eq:BSDE_H}. 
\end{Prop}
\begin{proof}
Here we do not construct $(H,Z^H)$ from scratch, but we use the existence of a minimal solution $(Y,Z^Y)$ of \eqref{eq:sing_BSDE}. If we define $H$ by \eqref{eq:asymp_beha_Y}, our previous computations show that:
 \begin{itemize}
\item The process $(H,Z^H)$ satisfies the dynamics given by \eqref{eq:dyn_BSDE_H} on any interval $[0,T-\eps]$. 
\item $H$ verifies an a priori estimate:
\begin{equation} \label{eq:a_priori_estim_H}
0\leq H_t \leq \dfrac{\vartheta(T-t)}{-\phi'(A_t)}   \leq \dfrac{\vartheta(T-t)}{-\phi'((T-t)/\etamin)}. 
\end{equation}
Hence $H$ is in $\bD^\infty(0,T-)$. 
\item $Z^H$ belongs to $\bH^\infty(0,T-)$.
\end{itemize}
Thus we only have to extend the assertions on $[0,T]$. Corollary \ref{lem:integrability_f_H_0} and Lemma \ref{lem:integrable_0_T} imply that the processes $F^H(\cdot,0,0)$ and $\dfrac{- f'(\phi(A))}{\eta}  H$ are integrable on $\Omega \times [0,T]$. 
Let us control the term:
$$ \left[ \frac{\kappa^1_t \kappa^2_t}{2}  \left( \frac{Z^A_t}{A_t} \right)^2 H_t + \kappa^1_t \frac{Z^A_t}{A_t} Z^H_t \right].$$
We already know that $\kappa^1$ and $\kappa^2$ are bounded and that $Z^A$ satisfies the inequality \eqref{eq:estim_ZA}. Let us precise the relation between $Z^H$ and $Z^Y$. Since $Y=\phi(A) -\phi'(A) H$, we have:
$$Z^Y_t = \phi'(A_t) Z^A_t - \phi^{(2)}(A_t) H_t Z^A_t  - \phi'(A_t) Z^H_t = \phi'(A_t) \left[ Z^A_t -f'(\phi(A_t)) H_t Z^A_t  -  Z^H_t\right] $$
thus
$$Z^H_t = \left[ 1 - f'(\phi(A_t)) H_t \right] Z^A_t + \frac{1}{\phi'(A_t)} Z^Y_t$$
and
\begin{eqnarray} \nonumber
\frac{Z^H_t}{(A_t)^{(1-\rho)/2}} &= &\left[ 1 -f'(\phi(A_t))H_t\right]  \frac{Z^A_t}{(A_t)^{(1-\rho)/2}} +  \frac{1}{\phi'(A_t)(A_t)^{(1-\rho)/2}} Z^Y_t \\ \label{eq:link_ZH_ZY}
& = & \left[ (A_t)^{\rho} + \kappa^1_t \dfrac{H_t}{(A_t)^{1-\rho}}\right]  \frac{Z^A_t}{(A_t)^{(1+\rho)/2}} +  \frac{1}{\phi'(A_t)(A_t)^{(1-\rho)/2}} Z^Y_t .
\end{eqnarray}

Equation \eqref{eq:lower_Y_estimate} and the concavity of $f$ leads to
$$-f'(Y_s) \geq -f'(\phi(A_s)) = \frac{\kappa^1_s }{A_s}.$$
Using Lemma \ref{lem:control_vartheta}, Estimates \eqref{eq:a_priori_estim_Y} \eqref{eq:bound_A},  and  \eqref{eq:increment_deriv_phi}  leads to
$$\left( \dfrac{\etamin}{(1+\eps)\etamax} \right)^{2K} \dfrac{1}{\left(\phi' \left(A_t \right)\right)^2} \leq \dfrac{1}{\left(\phi' \left( \frac{T-t}{(1+\eps)\etamax}\right)\right)^2}  \leq  \dfrac{1}{(f(\vartheta(T-t)))^2} \leq \dfrac{1}{(f(Y_t))^2}.$$
Using Lemma \ref{lem:estim_ZY}, for any $p> 1$,
$$\bE \left[\left(  \int_0^T \frac{\kappa^1_t (Z^Y_t)^2}{(\phi'(A_t))^2 A_t}dt\right)^p \ \right] \leq \bE \left[\left(  \int_0^T \frac{-f'(Y_t) (Z^Y_t)^2}{(f(Y_t))^2 }dt\right)^p \ \right] < +\infty.$$
Combining \eqref{eq:estim_ZA}, 
\eqref{eq:link_ZH_ZY} together this last estimate, we have
%
\begin{equation} \label{eq:integrable_sing_2}
\bE \int_0^T \left[\dfrac{\kappa^1_t \kappa^2_t}{2}  \left( \frac{Z^A_t}{A_t} \right)^2 H_t +\left| \kappa^1_t \frac{Z^A_t}{A_t} Z^H_t \right| \right] < +\infty
\end{equation}
if we can prove that for some $p>1$
$$\bE \left[ \left( \sup_{t \in [0,T]} \frac{H_t}{(A_t)^{1-\rho}}\right)^p\ \right] < +\infty.$$
Arguing as in the proof of Lemma \ref{lem:integrable_0_T}, we know that 
$$0\leq \frac{H_t}{(A_t)^{1-\rho}}\leq (\etamax)^{1-\rho} \frac{H_t}{(T-t)^{1-\rho}} \leq  (\etamax)^{1-\rho}  \left(\dfrac{2\etamax}{\etamin} \right)^K  \frac{\phi((T-t)/(2\etamax))}{-\phi'((T-t)/(2\etamax))(T-t)^{1-\rho}}.$$
The last term is deterministic and if we prove that this term remains bounded on $[0,T]$, the result follows. Note that 
\begin{eqnarray*} 
\frac{\phi((T-t)/(2\etamax))}{-\phi'((T-t)/(2\etamax))(T-t)^{1-\rho}} &=& \frac{\phi(x)}{-(2\etamax)^{1-\rho}\phi'(x)x^{1-\rho}} =  \frac{\phi(x)}{-(2\etamax)^{1-\rho}f(\phi(x))x^{1-\rho}} \\
&=& \frac{y}{-(2\etamax)^{1-\rho}f(y) G(y)^{1-\rho}} =\frac{y}{(2\etamax)^{1-\rho}h(y)} .
\end{eqnarray*}
with $x=(T-t)/(2\etamax)$ and $y=\phi^\star(x)$. From Condition \ref{C4_tech}, we deduce that \eqref{eq:integrable_sing_2} holds. 

Up to now, we have proved that 
$$\bE \left[ \int_0^T \left| L(s,H_s,Z^H_s) \right| + \dfrac{-f'(\phi(A_s))}{\eta_s} H_s ds \right] + < +\infty.$$
From \eqref{eq:dyn_BSDE_H}, for any $0 < t < T$
\begin{equation*}
\bE(H_0) = \bE(H_t) +\bE  \int_0^t F^H(s,H_s,Z^H_s) ds.
\end{equation*}
The a priori estimate \eqref{eq:a_priori_estim_H} and the preceding arguments imply that there exists some constant $C$ such that 
$$0\leq H_t \leq C \left( \frac{ \phi}{-\phi'} \right) \left( \frac{T-t}{2\etamax} \right) $$
and again we know that $y\mapsto y/(-f(y))$ tends to zero at infinity. Thus 
$$\lim_{t \to T} \bE (H_t) = 0.$$ 
Recall that
\begin{eqnarray*}
F^H(t,H_t,Z^H_t) &= &\frac{-1}{\eta_t \phi'(A_t)} \left[ f(\phi(A_t)-\phi'(A_t) H_t) - f(\phi(A_t)) \right] \\
 & - &\frac{1}{\eta_t } f'(\phi(A_t)) H_t  + L(t,H_t,Z^H_t) 
\end{eqnarray*}
and since $f$ is non increasing, 
$$\frac{1}{\eta_t \phi'(A_t)} \left[ f(\phi(A_t)-\phi'(A_t) H_t) - f(\phi(A_t)) \right]  \geq 0.$$
We deduce that 
$$\bE \left[ \int_0^T \frac{1}{\eta_t \phi'(A_t)} \left[ f(\phi(A_t)-\phi'(A_t) H_t) - f(\phi(A_t)) \right]  ds \right] < +\infty.$$
Therefore, again from \eqref{eq:dyn_BSDE_H}, for any $0 < t <  T$
\begin{equation*}
H_0 = H_t + \int_0^t F^H(s,H_s,Z^H_s) ds - \int_0^t Z^H_s dW_s,
\end{equation*}
and we deduce that 
$$\bE \left[ \sup_{0 < t < T} \left|  \int_0^t Z^H_s dW_s\right| \right] < +\infty.$$
By the Burkholder-Davis-Gundy inequality, we deduce that $Z^H$ is an element of $\bH^1(0,T)$, that is $(H,Z^H)$ is a solution of the BSDE \eqref{eq:BSDE_H} on $[0,T]$.
\end{proof} 

Note that we do not use the condition \ref{H_diff} on $\eta$. The additional assumption \ref{H_diff} supposed in Theorem  \ref{thm:uniq_sing_BSDE} is sufficient to ensure uniqueness\footnote{Without \ref{H_diff}, even the existence of a minimal solution for the BSDE \eqref{eq:BSDE_H} is unclear.} of the solution $(H,Z^H)$.

\begin{Prop} \label{prop:sing_gene_uniqueness}
Under {\rm \ref{A1}} to {\rm \ref{A3}}, {\rm \ref{C1_concave}} to {\rm \ref{C4_tech}} and {\rm \ref{H_diff}}, there exists a unique process $(H,Z^H)$ solution (in the sense of Definition \ref{def:sing_gene_sol}) of the BSDE \eqref{eq:BSDE_H}. 
\end{Prop}
\begin{proof}
Let us show first that any solution $(\widehat H,\widehat Z)$ of \eqref{eq:BSDE_H} is non-negative. From the It\^o formula for the non-positive part of $\widehat H$ we obtain for $t \in [\widehat \tau,T]$:
\begin{eqnarray*}
\left(\widehat H _t \right)^-  &\leq & - \int_t^{T} F^H(s,\widehat H_s,\widehat Z_s) \mathbf 1_{\widehat H_s\leq 0} ds + \int_t^{T} (\widehat Z_s)  \mathbf 1_{\widehat H_s\leq 0} dW_s \\
& = &  - \int_t^T \left[ \frac{\lambda_s}{-\phi'(A_s)} +\frac{ \kappa^1_s}{2} \frac{(Z^A_s)^2}{A_s}+ \frac{\kappa^1_s \kappa^2_s}{2}  \left( \frac{Z^A_s}{A_s} \right)^2 \widehat H_s + \kappa^1_s \frac{Z^A_s}{A_s} \widehat Z_s\right]   \mathbf 1_{\widehat H_s\leq 0}  ds \\
&& \quad + \int_t^{T} (\widehat Z_s)  \mathbf 1_{\widehat H_s\leq 0} dW_s \\
& \leq & \int_t^T\frac{ \kappa^1_s\kappa^2_s}{2}  \left( \frac{Z^A_s}{A_s} \right)^2 \left( \widehat H_s \right)^- \, ds 
-\int_t^T  \kappa^1_s \frac{Z^A_s}{A_s} \widehat Z_s   \mathbf 1_{\widehat H_s\leq 0}  ds + \int_t^{T} (\widehat Z_s)  \mathbf 1_{\widehat H_s\leq 0} dW_s.
\end{eqnarray*}
From the condition \ref{H_diff}, Lemma \ref{lem:expo_moment_1} and the Novikov's criterion, the martingale 
\begin{equation} \label{eq:loc_mart}
\cE(Z^A)_t = \exp \left( \int_{\widehat T}^t \kappa^1_s \frac{Z^A_s}{A_s} dW_s + \frac{1}{2} \int_{\widehat T}^t (\kappa^1_s)^2 \left( \frac{Z^A_s}{A_s} \right)^2 ds \right),\qquad t \in [\widehat T,T],
\end{equation}
is uniformly integrable. Using Girsanov's theorem and the expression of the solution of a linear BSDE (see \cite[Proposition 5.31]{pard:rasc:14}), we get that a.s. for any $t\in [\widehat T,T]$, $\left(\widehat H _t \right)^- =0$. 

%
%

Now we prove uniqueness of the solution. Consider $(\widehat H,\widehat Z)$ is another solution of \eqref{eq:BSDE_H}. Since $f$ is concave, we have for any $h\geq 0$ and $h'\geq 0$:
\begin{eqnarray*} 
&& F(t,h) - F(t,h')= \frac{-1}{\eta_t \phi'(A_t)} \left[ f(\phi(A_t)-\phi'(A_t) h) - f(\phi(A_t)) +f'(\phi(A_t)) \phi'(A_t) h \right]  \\
&&\qquad  \frac{1}{\eta_t \phi'(A_t)} \left[ f(\phi(A_t)-\phi'(A_t) h') - f(\phi(A_t)) +f'(\phi(A_t)) \phi'(A_t) h' \right] \\
&&\ =  \frac{-1}{\eta_t \phi'(A_t)} \left[ f(\phi(A_t)-\phi'(A_t) h)-f(\phi(A_t)-\phi'(A_t) h')  + f'(\phi(A_t)-\phi'(A_t) h') \phi'(A_t) (h-h')\right]  \\
&&\qquad + \frac{ f'(\phi(A_t)-\phi'(A_t) h')-f'(\phi(A_t))}{\eta_t } \left(   h -  h' \right) \\
&&\ \leq \frac{ f'(\phi(A_t)-\phi'(A_t) h')-f'(\phi(A_t))}{\eta_t } \left(   h -  h' \right).
 \end{eqnarray*}
Hence if $\Delta H = \widehat H - H$, $\Delta Z = \widehat Z - Z^H$, then
\begin{eqnarray*}
(\Delta H _t)^+  &\leq & 
 \int_t^{T} \frac{1}{\eta_s} \left\{f'(\phi(A_s)-\phi'(A_s) H_s)-f'(\phi(A_s)) \right\} ( \Delta H_s )^+  ds \\
& + &  \int_t^T \left[ \frac{\kappa^1_s \kappa_s^2}{2}  \left( \frac{Z^A_s}{A_s} \right)^2 \Delta H_s + \kappa^1_s\frac{Z^A_s}{A_s}  \Delta Z^H_s \right] \mathbf{1}_{\Delta H_s \geq 0} \,ds- \int_t^T \Delta Z^H_s  \mathbf{1}_{\Delta H_s \geq 0}dW_s.
\end{eqnarray*}
But $f'$ is non-increasing, $H\geq 0$ and $-\phi' \geq 0$. Hence
$$(\Delta H _t)^+ \leq   \int_t^T \left[ \frac{\kappa^1_s \kappa_s^2}{2}  \left( \frac{Z^A_s}{A_s} \right)^2 (\Delta H_s)^+ + \kappa^1_s\frac{Z^A_s}{A_s}  \Delta Z^H_s \mathbf{1}_{\Delta H_s \geq 0}\right]  \,ds- \int_t^T \Delta Z^H_s  \mathbf{1}_{\Delta H_s \geq 0}dW_s.$$
Arguing as before yields that $(\Delta H)^+$ is equal to zero, that is $\widehat H \leq H$. Then uniqueness holds on $[\widehat T, T]$. 

Let us extend uniqueness on the whole time interval $[0,T]$. If $(\widehat H,\widehat Z)$ still denotes another solution, then the two processes solve the same BSDE \eqref{eq:BSDE_H} on $[0,\widehat T]$ with the same terminal condition $H_{\widehat T}=\widehat H_{\widehat T}$. Since the generator of \eqref{eq:BSDE_H} remains singular on the whole interval (due to the linear term), uniqueness on the rest of the time interval $[0,\widehat T]$ is not trivial. But if we define
$$\widehat Y_t = \phi(A_t) -  \phi'(A_t) \widehat H_t,$$
then $\widehat Y$ is the first part of the solution of the BSDE \eqref{eq:sing_BSDE} on $[0,\widehat T]$ with the bounded terminal condition $\phi(A_{\widehat  T}) - \phi'(A_{\widehat  T}) \widehat H_{\widehat  T}$. Since uniqueness holds for the BSDE \eqref{eq:sing_BSDE}, we deduce that $\widehat H=H$ also on $[0,\widehat T]$. 
\end{proof}

To finish the proof of Theorem \ref{thm:uniq_sing_BSDE}, we only need to prove uniqueness of the solution of the BSDE \eqref{eq:sing_BSDE} (in the sense of Definition \ref{def:sing_cond_sol}). Let us consider $(\widehat Y, \widehat Z)$ solution of the BSDE \eqref{eq:sing_BSDE}. Then it satisfies \eqref{eq:a_priori_estim_Y} and the property of Lemma \ref{lem:estim_ZY}. Therefore if we define $\widehat H = (\widehat Y - \phi(A))/(-\phi'(A))$, then this process $\widehat H$ and the related $\widehat Z^H$ solve the BSDE \eqref{eq:BSDE_H} (in the sense of Definition \ref{def:sing_gene_sol}). From the previous proposition, $\widehat H = H$, and thus $\widehat Y = Y$.  This achieves the proof of our main result. 

\begin{Rem}
In fact we can replace {\rm \ref{H_diff}} by any condition such that the local martingale defined by \eqref{eq:loc_mart} 
is a uniformly integrable martingale on some interval $[\widehat T, T]$ and such that we can apply Girsanov's theorem. 
\end{Rem}

\subsection{Asymptotic behaviour} \label{ssect:asymp_behaviour}

From \eqref{eq:a_priori_estim_Y}, $-\phi'(A_t) H_t \leq \vartheta(T-t)$ and we already proved that on some interval $[\widehat T,T]$ , 
$$\frac{-\phi'(A_t) H_t }{\phi(A_t)} \leq \frac{\phi \left( (T-t)/2\etamax\right)}{\phi \left( (T-t)/\etamin\right)}.$$
\begin{Lem} \label{lem:asymp_boundedness}
There exists a constant $C_{\eta}$ depending on $\etamax$ and $\etamin$ such that for any $t \in [\widehat T,T]$,
$$0\leq \frac{-\phi'(A_t) H_t }{\phi(A_t)}\leq C_{\eta}.$$
\end{Lem}
\begin{proof}
Recall that since $\phi$ is non increasing, $\phi(y) \leq \phi(ay)$. Moreover with \eqref{eq:increment_deriv_phi}, from the very definition of $-\phi'$, this inequality can be written as: $- \phi'(ay) \leq -a^{-K}  \phi'(y)$. 
Integrating this inequality (between $y$ and $\eta$) leads to:
$$ \phi(ay) \leq a^{1-K} \phi(y) + C$$
for some constant $C \geq 0$. Hence
$$1 \leq \frac{\phi(ay)}{\phi(y)} \leq a^{1-K} + \frac{C}{\phi(y)}.$$
Since $\phi(0)=\infty$, the conclusions follows from this inequality.
\end{proof}

Hence we have proved that the minimal solution $Y$ of the BSDE \eqref{eq:sing_BSDE} satisfies a.s. on the interval $[\widehat T,T]$:
$$\phi(A_t) \leq Y_t \leq \phi(A_t) (1+ C_\eta).$$

Furthermore if $\eta$ or $1/\eta$ is an It\^o process (condition \ref{H_diff}), then using \eqref{eq:decomp_A_Ito_case}:
\begin{eqnarray*} 
\phi(A_t) & = & \phi \left(\gamma_t (T-t) + \bE \left[  \int_t^T (T-u) b^\gamma_u du  \bigg| \cF_t \right]\right) \\
& = & \phi \left(\gamma_t (T-t) \right) \\
& + & \bE \left[  \int_t^T (T-u) b^\gamma_u du  \bigg| \cF_t \right]  \int_0^1 \phi' \left(\gamma_t (T-t) + a \bE \left[  \int_t^T (T-u) b^\gamma_u du  \bigg| \cF_t \right]\right) da.
\end{eqnarray*}
With \eqref{eq:increment_deriv_phi} and with Condition \ref{A1} and for a bounded process $b^\gamma$, we deduce that there exists a constants $C$ such that 
$$\frac{1}{C} \phi'(T-t) \leq \phi' \left(\gamma_t (T-t) + a \bE \left[  \int_t^T (T-u) b^\gamma_u du  \bigg| \cF_t \right]\right) \leq C \phi'(T-t).$$
Thus 
$$\phi(A_t) =  \phi (\gamma_t (T-t)) - \phi'(T-t) \kappa_t (T-t)^2$$
with a bounded process $\kappa$. And using \eqref{eq:bound_A} to control $\phi'(A_t)$, the asymptotic development \eqref{eq:asymp_beha_Y} becomes: for some bounded process $\widehat \kappa$:
$$Y_t = \phi\left(\dfrac{T-t}{\eta_t}\right) - \phi'(T-t)\widehat \kappa_t \left[ (T-t)^2 + H_t\right].$$
%

 \section{Some extensions} \label{sect:general_gene}

\subsection{The power case $f(y)=-y|y|^q$} \label{ssect:power_case}

In this case recall that $\phi(x) = \left( \dfrac{1}{qx}\right) ^{1/q}$ and $-\phi'(x) = \left( \dfrac{1}{qx}\right) ^{1+1/q}$ and we assume that $\eta$ is an It\^o process,
\begin{equation} \label{eq:eta-ito}
d\eta_t=b^\eta_t\,dt+\sigma^\eta_t\,dW_t,
\end{equation}
such that $b^\eta \in \bL^\infty(\Omega\times [0,T] ;\mathbb R)$ and $\sigma^\eta \in \bL^2(\Omega\times [0,T];\mathbb R^d)$.
As mentioned as the end of the preceding section, we consider 
$$ \phi\left( \frac{T-t}{\eta_t}\right)=\left( \frac{\eta_t}{q(T-t)}\right)^{1/q} = \frac{\zeta_t}{(q(T-t))^{1/q}} =\zeta_t \phi(T-t)$$ 
and again from condition \ref{A1}, $\zeta$ is an It\^o process with drift $b^\zeta\in  L^\infty([0,T]\times \Omega;\mathbb R)$ and diffusion matrix $\sigma^\zeta \in L^2([0,T]\times \Omega;\mathbb R^d)$. 
We replace \eqref{eq:asymp_beha_Y} by \eqref{eq:asymp_beha_power_case}:
\begin{equation*}
Y_t=\zeta_t\phi(T-t)-\phi'(T-t) H^{\#}_t = \left( \eta_t \right)^{1/q} \phi(T-t)-\phi'(T-t) H^{\sharp}_t ,
\end{equation*}
and we obtain the dynamics for $H^{\sharp}$:
\begin{eqnarray}\nonumber 
-dH^{\sharp}_t&=&(-\phi'(T-t))^{-1}\left[ \lambda_t+\phi(T-t) b^\zeta_t \right] \, dt - Z^{H^\sharp}_t dW_t \\ \nonumber
&+& \dfrac{f(\zeta_t\phi(T-t)-\phi'(T-t)H^\sharp_t)-f(\zeta_t\phi(T-t))+f'(\zeta_t\phi(T-t))\phi'(T-t) H^\sharp_t}{(-\phi'(T-t)) \eta_t} \,dt \\  \label{eq:BSDE_H_power_case} 
&=& F^\sharp(t,H^\sharp_t)\,dt-Z^{H^\sharp}_tdW_t,
\end{eqnarray}
where $F^\sharp$ can be rewritten as 
\begin{align*}
F^\sharp(t,h)&=\frac{\lambda_t}{-\phi'(T-t)}+\frac{\phi(T-t)}{-\phi'(T-t)} b^\zeta_t \\
& +\frac{-\phi'(T-t)h^2}{\eta_t}\int_0^1f''(\zeta_t\phi(T-t)-a\phi'(T-t) h)(1-a)\,da,
\end{align*}
where $f''(y) =-q(q+1) |y|^{q-1}\mbox{sign}(y)$. Since $-\phi'(x) = \dfrac{\phi(x)}{qx} $ and $\phi(x)^q= 1/(qx)$, we obtain that 
\begin{align*}
F^\sharp(t,h) &\ =F^\sharp(t,0)-\frac{(q+1)h^2}{q\eta_t (T-t)^2}\int_0^1\left( \zeta_t+\dfrac{ah}{q(T-t)}\right)^{q-1}\mbox{sign}\left( \zeta_t+\frac{ah}{q(T-t)}\right)(1-a)\,da.
\end{align*}
Moreover 
$$F^\sharp(t,0) = q(T-t) \left[\lambda_t (q(T-t))^{1/q} + b^\zeta_t \right].$$
Let us remark that this generator is again {\it singular} and that the second derivative of $f$ is not well-defined at zero if $0< q < 1$.

To establish local existence for \eqref{eq:BSDE_H_power_case}, we don't use monotonicity arguments. But instead, we proceed very similar as in \cite{grae:hors:sere:18} and carry out the Picard iteration in the space $\mathcal H^\delta$ defined by \eqref{eq:def_cH_delta} and \eqref{eq:def_norm_cH_delta}. 
\begin{Lem}  \label{lem:bounded_generator}
Let $R>0$ and $\delta\in(0,(\etamin)^{1/q}/(qR))$ then for every $H\in\overline B_{\mathcal H^\delta}(R)$
we have $(F^\sharp(t,H_t))_{t\in[T-\delta,T]}\in L^\infty([T-\delta,T]\times \Omega;\mathbb R)$.
\end{Lem}
\begin{proof}
From our assumptions, 
$F^\sharp(\cdot,0)$ is bounded and thus in $L^\infty([0,T]\times \Omega;\mathbb R)$. By definition if  $H\in\overline B_{\mathcal H^\delta}(R)$, then a.s. for any $t\in [T-\delta,T]$
$$\left| \frac{q+1}{q\eta_t}\frac{H_t^2}{(T-t)^2}\right| \leq \frac{q+1}{q\etamin}R^2\delta^2 .$$
Note that $\delta\in(0,(\etamin)^{1/q}/(qR))$ ensures that $\zeta_t+a H_t/(q(T-t))>0$ for all $t\in [T-\delta,T]$, $a\in[0,1]$. And 
$$\int_0^1\left| \zeta_t+\dfrac{aH_t}{q(T-t)}\right|^{q-1}(1-a)\,da \leq \left( (\etamax)^{1/q} + \dfrac{R\delta}{q}\right)^{q-1} .$$
The lemma is now proved.
\end{proof}

The preceding lemma allows to define by
\[
\Gamma(H)=\left(\mathbb E\left[\left.\int_t^T F^\sharp(s,H_s)\,ds\right|\mathcal F_t\right]\right)_{t\in[T-\delta,T]}
\]
the operator $\Gamma:\overline B_{\mathcal H^\delta}(R)\to L^\infty(\Omega;C([T-\delta,T];\mathbb R))$.

\begin{Lem} \label{lemma-locally-Lip} For every $R>0$ there exists a constant $L>0$ independent of $\delta\in(0,\etamin^{1/q}/(qR))$ such that
\[
|F^\sharp(t,H_t)-F^\sharp(t,H_t')|\leq L|H_t-H_t'| \qquad \forall t\in[T-\delta,T]\ \forall H,H'\in\overline B_{\mathcal H^\delta}(R),\ a.s. 
\]
\end{Lem}
\begin{proof} 
We have for $q \neq 1$
\begin{align*}
\frac{dF^\sharp}{dH}(t,H)&=-\frac{2(q+1)}{q\eta_t}\frac{H}{(T-t)^2}\int_0^1\left( \zeta_t+\dfrac{aH}{q(T-t)}\right)^{q-1} (1-a)\,da\\
&\quad-\frac{(q+1)(q-1)\mathbf 1_{q\neq 1}}{q\eta_t}  \frac{H^2}{(T-t)^2}\int_0^1\left( \zeta_t+a\dfrac{H}{T-t}\right)^{q-2}\frac{a(1-a)}{q(T-t)}\,da .
\end{align*}
Hence, there exists $L>0$ such that 
\begin{align*}
\left\|\frac{dF^\sharp}{dH}(t,(T-t)^2R)\right\|_\infty\leq L \qquad \forall t\in[T-\delta,T].
\end{align*}
The assertion then follows by the mean value theorem.
\end{proof}

We are now ready to prove that $\Gamma$ maps $\overline B_{\mathcal H^\delta}(R)$ contractiv into itself (for appropriate $R$ and $\delta\in(0,\etamin^{1/q}/(qR))$). For $R>0$ specified below choose $L>0$ as in Lemma~\ref{lemma-locally-Lip}. For $H,H'\in\overline B_{\mathcal H^\delta}(R)$ it then holds for all $t\in[T-\delta,T]$
\begin{align*}
|\Gamma(H)_t-\Gamma(H')_t|&\leq \mathbb E\left[\left.\int_t^T|F^\sharp(s,H_s)-F^\sharp(s,H_s')|\,ds\right|\mathcal F_t\right]\\
\leq & (T-t)^3L \|H-H'\|_{\mathcal H^\delta}.
\end{align*}
This yields, as long as $0<\delta\leq 1/(2L)$,
\[
\|\Gamma(H)-\Gamma(H')\|_{\mathcal H^\delta}\leq \frac{1}{2}\|H-H'\|_{\mathcal H^\delta}.
\]
Hence, $\Gamma$ is an $1/2$-contraction on $\overline B_{\mathcal H^\delta}(R)$ if $\delta\leq 1/(2L)$. Furthermore, for $H\in\overline B_{\mathcal H^\delta}(R)$, 
\begin{align*}
|\Gamma(H)_t|&\leq |\Gamma(H)_t-\Gamma(0)_t|+|\Gamma(0)_t|\\
	&\leq (T-t)^2 \frac{R}{2}+\mathbb E\left[\left.\int_t^T \left[ (q(T-s))^{1+1/q}\lambda_s+q(T-s)|b^\zeta_s|\right] \,ds\right|\mathcal F_t\right]\\
	&\leq (T-t)^2 \frac{R}{2}+(T-t)^2(\delta^{1/q}q^{1+1/q}\lambdamax+q\|b^\zeta\|_\infty).
\end{align*}
Thus, choosing $R=2(q^{1+1/q}\lambdamax+q\|b^\zeta\|_\infty)$ and $\delta=\min\{1,1/2L,\etamin^{1/q}/(qR)\}$ yields $\|\Gamma(H)\|_{\mathcal H^\delta}\leq R$.

\begin{Thm} \label{thm:power_case}
If $\eta$ is given by \eqref{eq:eta-ito}, with $b^\eta$ bounded and $\sigma^\eta$ in $L^2$, then the BSDE \eqref{eq:BSDE_H_power_case} has a unique solution $(H^\sharp,Z^{H^\sharp})$ on $[0,T]$ such that:
$$\left\|\sup_{t\in{[0,T)}}(T-t)^{-2}|H^\sharp_t|\right\|_\infty < +\infty.$$
Moreover $\int_{0}^\cdot Z^{H^\sharp} dW$ is a BMO-martingale. The relation \eqref{eq:asymp_beha_power_case} and the uniqueness for the BSDE \eqref{eq:sing_BSDE} hold. 
\end{Thm}
\begin{proof}
Using the property of the map $\Gamma$, we deduce that there exists $\delta> 0$ such that there exists a unique process $H^\sharp  \in \mathcal H^\delta$ such that a.s. for any $t \in [T-\delta,T]$:
$$H^\sharp_t = \mathbb E\left[\left.\int_t^T F^\sharp(s,H^\sharp_s)\,ds\right|\mathcal F_t\right].$$
By the martingale representation, we obtain $Z^{H^\sharp}$ and since $H^\sharp \in \mathcal H^\delta$, from Lemma \ref{lem:bounded_generator}, we deduce that the martingale $\int_{T-\delta}^\cdot Z^{H^\sharp} dW$ is a BMO martingale. 

In particular the random variable $H^\sharp_{T-\delta}$ is bounded. If we consider the BSDE \eqref{eq:BSDE_H_power_case} starting at time $T-\delta$ from the terminal condition $H^\sharp_{T-\delta}$, we can apply directly \cite[Proposition 5.24]{pard:rasc:14} to obtain a unique solution $(H^\sharp,Z^{H^\sharp})$ on $[0,T-\delta]$ such that $H^\sharp$ is bounded. 
\end{proof}

\subsection{Non concave case}

The assumptions \ref{C1_concave} and \ref{C3_tech} are supposed to ensure the boundedness of the coefficient $\kappa^2$ in the generator $F^H$. To avoid this difficulty, we can modify the relation \eqref{eq:asymp_beha_Y}  and consider \eqref{eq:asymp_beha_general_case}, that is:
\begin{equation*} 
Y_t = \phi(A_t) - \phi' \left(\frac{T-t}{\etamax} \right)\widehat H_t =\phi(A_t) - \psius(t) \widehat H_t.
\end{equation*}
Again from Lemma \ref{lem:lower_bound}, we know that $Y_t \geq \phi(A_t)$, thus $\widehat H_t \geq 0$ a.s. And we deduce
\begin{equation} \label{eq:BSDE_H_general_case}
\widehat H_t  =  \int_t^T F^{\widehat H}(s,\widehat H_s) ds - \int_t^TZ^{\widehat H}_s dW_s,
\end{equation}
where
\begin{align*}
F^{\widehat H}(t,h) & =  \frac{1}{-\psius(t) \eta_t}\left[  f(\phi(A_t) - \psius(t) h) - f (\phi(A_t)) \right] \mathbf{1}_{h\geq 0}+\frac{\kappa^3_t}{T-t}h \\
& + \left[ \frac{\lambda_t}{-\psius(t)}  + \frac{\widehat\kappa^1_t}{2}\frac{(Z^A_t)^2}{A_t}\right] 
\end{align*}
with
$$\widehat\kappa^1_t = \frac{ \phi' \left( A_t\right)}{\psius(t)}\kappa^1_t \geq 0, \qquad  \kappa^3_t = - \frac{T-t}{\etamax\psius(t)}\phi^{(2)} \left( \frac{T-t}{\etamax} \right)\geq 0.$$
The estimates \eqref{eq:bound_A} together with Lemma \ref{lem:bounded_cond} lead to the boundedness of the coefficients $\widehat\kappa^1$ and $\kappa^3$.
\begin{Thm} \label{thm:one_one_correspondance}
Under Conditions {\rm \ref{A1}} to {\rm \ref{A3}} and {\rm \ref{C2_tech}}, 
the minimal solution $(Y,Z^Y)$ of the BSDE \eqref{eq:sing_BSDE} is given by \eqref{eq:asymp_beha_general_case}, where $(\widehat H,Z^{\widehat H})$ is the minimal solution of the BSDE \eqref{eq:BSDE_H_general_case}.
\end{Thm}
\begin{proof}
Define $\widehat H$ on $[0,T)$ thanks to \eqref{eq:asymp_beha_general_case} and set $\widehat H_T =0$ a.s. 
The same arguments as for Proposition \ref{prop:existence_sol_BSDE_sing_gen} imply that $(\widehat H,Z^{\widehat H})$ is solution to the BSDE \eqref{eq:BSDE_H_general_case} (in the sense of Definition \ref{def:sing_gene_sol}). It is even easier since $F^{\widehat H}$ does not depend on $Z^{\widehat H}$. 
 
Assume that the BSDE  \eqref{eq:BSDE_H_general_case} has another solution $( \widetilde H,Z^{ \widetilde H})$. The first part of the proof of Proposition \ref{prop:sing_gene_uniqueness} with straightforward modifications shows that $ \widetilde H_t \geq 0$ a.s. We define $\widetilde Y_t = \phi(A_t) - \psius(t) \widetilde H_t$. Thus $(\widetilde Y,\widetilde Z)$ is a non-negative solution of the BSDE \eqref{eq:sing_BSDE}. Since $(Y,Z)$ is the minimal non-negative solution of \eqref{eq:sing_BSDE} (see Proposition \ref{prop:exist_min_sol}), we have a.s. 
$$\forall t \in [0,T], \quad \phi(A_t) \leq Y_t \leq \widetilde Y_t = \phi(A_t) -\psius(t) \widetilde H_t.$$
Thus $0\leq \widehat H_t  \leq \widetilde H_t$. In other words $\widehat H$ is the minimal solution of  \eqref{eq:BSDE_H_general_case}. 
\end{proof}

In the expansion \eqref{eq:asymp_beha_general_case} of $Y$, there is an asymmetry between $\phi(A_t)$ which is random, and the deterministic $\psius(t)$. This asymmetry has the advantage of avoiding the presence of $Z^H$ in the generator of $H$ and of an extra term with the third derivative of $\phi$. However it leads to the fact that 
 we cannot interpret the bracket
\begin{align*}
&\frac{\kappa^3_t}{T-t}h+\frac{1}{-\psius(t) \eta_t}[f(\phi(A_t)-\psius(t) h)-f(\phi(A_t))]\\
&=\frac{1}{-\psius(t) \eta_t}\left[f(\phi(A_t)-\psius(t) h)-f(\phi(A_t))+f'\left(\phi\left(\frac{T-t}{\etamax}\right)\right)\eta_t\psius(t)h\right]
\end{align*}
as the reminder to the first Taylor polynomial of $f$ at $\phi(A_t)$. And without the other hypotheses \ref{C1_concave}, \ref{C3_tech}, \ref{C4_tech} and \ref{H_diff}, we cannot prove uniqueness of the solution.

\section{Summary}

Let us summarize our results. We suppose that \ref{A1} to \ref{A3} hold. 
\begin{itemize}
\item Under the assumptions \ref{C1_concave} to \ref{C4_tech} and \ref{H_diff}, $Y$ can be developed as follows:
$$Y_t = \phi(A_t) - \phi' \left(A_t \right) H_t $$
where $H$ is the unique solution of the BSDE with singular generator \eqref{eq:BSDE_H}. Hence uniqueness of $Y$ holds. 


\item In the power case $f(y) = -y|y|^q$ and if $1/\eta$ is given by \eqref{eq:gamma_ito_proc}, we can use \eqref{eq:asymp_beha_power_case}:
$$Y_t=\left(  \dfrac{ \eta_t}{ q (T-t)} \right)^{1/q} +\frac{1}{(q(T-t))^{1+1/q}} H^{\sharp}_t,$$
where $H^\sharp$ solves the BSDE \eqref{eq:BSDE_H_power_case} .

\item If we only have \ref{C2_tech}, $Y$ and $\widehat H$ are related by \eqref{eq:asymp_beha_general_case}:
$$Y_t = \phi(A_t) - \phi' \left(\frac{T-t}{\etamax} \right)\widehat H_t$$
where $\widehat H$ is the minimal solution of the BSDE \eqref{eq:BSDE_H_general_case}. 
\end{itemize}
Remark that if $\eta$ or $1/\eta$ is an It\^o process, then using \eqref{eq:decomp_A_Ito_case}:
\begin{eqnarray*} 
\phi(A_t) & = & \phi \left(\gamma_t (T-t) + \bE \left[  \int_t^T (T-u) b^\gamma_u du  \bigg| \cF_t \right]\right) \\
& = & \phi \left(\gamma_t (T-t) \right) \\
& + & \bE \left[  \int_t^T (T-u) b^\gamma_u du  \bigg| \cF_t \right]  \int_0^1 \phi' \left(\gamma_t (T-t) + a \bE \left[  \int_t^T (T-u) b^\gamma_u du  \bigg| \cF_t \right]\right) da.
\end{eqnarray*}
From Condition \ref{A1}, \eqref{eq:increment_deriv_phi} and for a bounded process $b^\gamma$, we deduce that there exists a constants $C$ such that 
$$-\frac{1}{C} \phi'(T-t) \leq -\phi' \left(\gamma_t (T-t) + a \bE \left[  \int_t^T (T-u) b^\gamma_u du  \bigg| \cF_t \right]\right) \leq -C \phi'(T-t).$$
Thus 
$$\phi(A_t) =  \phi (\gamma_t (T-t))  -\phi'(T-t) \kappa_t (T-t)^2$$
with a bounded process $\kappa$. In other words, in the It\^o setting, all developments coincide: we can find some constant $C > 1$ such that a.s. for any $t \in [0,T]$
$$\frac{1}{C} \widehat H_t \leq  H_t \leq C \widehat H_t, \quad \frac{1}{C} H^\sharp_t \leq  H_t \leq C H^\sharp_t.$$ 

Moreover from the construction of $H^\sharp$, we know that $|H^\sharp_t | \leq C(T-t)^2$. Using our different asymptotics, the previous development of $\phi(A)$ and uniqueness of the (minimal) solution, we obtain that $H$, $\widetilde H$ and $\widehat H$ verify also this estimate, which is better than \eqref{eq:a_priori_estim_H}. 

\section{Appendix} \label{sect:appendix}

We add here some additional results. The first one concerns the process $Z^A$. Under \ref{H_diff}, it is a corollary of Lemma \ref{lem:expo_moment_1}. 
\begin{Lem}\label{lem:estim_F_0_bis}
For $p>2$, if $1/\eta$ is given by \eqref{eq:gamma_ito_proc} and if $b^\gamma$ and $\sigma^\gamma$ belong to $\bL^{2p}(\Omega\times [0,T])$, the process $Z^A/A$ belongs to $\bH^{2p}(0,T)$, that is
\begin{equation} \label{mbappe}
\bE \left[ \left(  \int_0^T  \left(\frac{Z^A_s}{A_s}\right)^{2}  ds \right)^{p} \right] < +\infty.
\end{equation}
\end{Lem}
\begin{proof} We consider the process $\bar A_t:=A_t/(T-t)$, which satisfies the BSDE
\[
-d\bar A_t=\frac{1/\eta_t-\bar A_t}{T-t}\,dt-Z^{\bar A}\,dW_t, \qquad \bar A_T=1/\eta_T.
\]
Since $Z^{\bar A}_t=Z^{ A}_t/(T-t)$, to verify \eqref{mbappe} it is sufficient to establish $Z^{\bar A}\in \bH^{2p}(0,T)$. For the later again it is sufficient to establish that the driver to $\bar A$ is in $L^{2p}$. (Here we used frequently that $\eta$ is bounded above and away from zero.)

To establish $(1/\eta-\bar A)(T-\cdot)\in L^{2p}$ we first check Kolgomorov's criterion for $1/\eta$: For $0\leq t\leq s\leq T$, by Jensen and BDG inequality,
\[
\mathbb E[|1/\eta_s-1/\eta_t|^{2p}]\leq C|s-t|^{p-1}\mathbb E\left[\int_t^s \left( |b^\gamma_r|^{2p}+|\sigma^\gamma_r|^{2p} \right) \,dr\right].
\]
Hence, by Kolgomorov's criterion, for any $\alpha\in(0,\frac{p-2}{2p})$ there exits a random variable $\xi\in L^{2p}(\Omega)$ such that
\[
|1/\eta_t-1/\eta_s|\leq \xi|t-s|^{\alpha}, \qquad t,s\in[0,T].
\]
Therefore, using the mean value theorem,
\[
\mathbb E\left[\left( \int_0^T\left|\frac{1/\eta_t-\bar A_t}{T-t}\right|dt \right)^{2p}\,\right]\leq\mathbb E\left[\left(\int_0^T\frac{\xi(T-t)^\alpha}{T-t}dt \right)^{2p}\,\right]\leq C \mathbb E[\xi^{2p}],
\]
which completes the proof.
\end{proof}

The coefficient in the linear part $h\mapsto L(t,h,z)$ of the BSDE \eqref{eq:BSDE_H} is in $\bH^{2p}(0,T)$. However it is not sufficient to control $z \mapsto L(t,h,z)$. 

\subsection{Non-negativity of $\lambda$} \label{ssect:sign_lambda}

From the comparison principle for monotone BSDE (see \cite[Proposition 5.34]{pard:rasc:14}), any solution of \eqref{eq:sing_BSDE} with a non-negative terminal condition is bounded from below by the solution $(\bar Y, \bar Z)$ of the BSDE with generator 
$$f_\star(\omega,t,y) = \dfrac{1}{\eta_t(\omega)} (f(y)-f(0)) - (f(0)+\lambda_t(\omega))^-$$
and terminal condition 0. $\bar Y$ is non-positive and if $\lambda$ is bounded, $\bar Y$ is also bounded. Thus the negative part of $Y$ is bounded and we can consider only the positive part of the solution. 

If the sign of $\lambda$ is unknown, then Lemma \ref{lem:lower_bound} does not hold. However the minimal solution of \eqref{eq:sing_BSDE} is bounded from below by the minimal solution $(Y_\star,Z^{Y_\star})$ of the BSDE with generator $f_\star$ and terminal condition $+\infty$. And we can adapt the proof of Lemma \ref{lem:gene_a_priori_estim} in order to prove that there exist two functions $\vartheta_\star$ and $\vartheta^\star=\vartheta$ such that:
$$\vartheta_\star(T-t) \leq (Y_\star)_t \leq Y_t \leq \vartheta^\star(T-t),$$
where $\vartheta_\star$ is the solution of the ODE:
$$y' = \lambda_\star - \frac{f(y)}{\etamin}$$
with $\lambda_\star  \leq f(0)+\lambda_t(\omega) \leq \lambdamax$ and $\vartheta_\star(0) = +\infty$. Arguing as in the proof of Lemma \ref{lem:control_vartheta} we get that for any $0\leq \eps < 1$, on some deterministic and non-empty interval $[T^\eps,T]$, a.s.
$$\phi \left( \frac{T-t}{(1-\eps)\etamin}\right) \leq Y_t \leq \phi \left( \frac{T-t}{(1+\eps)\etamax}\right) .$$

\subsection{Construction of the process $H$} \label{ssect:constr_H}

Our aim is to prove that the BSDE \eqref{eq:BSDE_H} has a solution $(H,Z^H)$, without using the existence of $Y$. In other words we want to construct $(H,Z^H)$ from scratch. 
Recall that the generator $F^H$ is given by $F^H = F+ L$ with \eqref{eq:gene_F} and \eqref{eq:gene_L}
\begin{eqnarray*}
F(t,h) &= &\frac{-1}{\eta_t \phi'(A_t)} \left[ f(\phi(A_t)-\phi'(A_t) h) - f(\phi(A_t)) +f'(\phi(A_t)) \phi'(A_t) h \right]  \mathbf 1_{h\geq 0}, \\ 
L(t,h,z) & = &- \frac{\lambda_t}{\phi'(A_t)} + \frac{\kappa^1_t}{2} \frac{(Z^A_t)^2}{A_t} + \frac{ \kappa^1_t \kappa^2_t}{2}  \left( \frac{Z^A_t}{A_t} \right)^2 h + \kappa^1_t \frac{Z^A_t}{A_t} z \\
& =& \varpi_t + \frac{ \kappa^1_t \kappa^2_t}{2}  \left( \frac{Z^A_t}{A_t} \right)^2 h + \kappa^1_t \frac{Z^A_t}{A_t} z.
\end{eqnarray*}
The main properties of $F^H$ are summarized at the beginning of Section \ref{ssect:generator_f_H}. 
Note that we cannot apply directly the results of \cite{pard:rasc:14}, due to the singularity of $F^H$ and the cases of BSDEs with singular generator studied in \cite{jean:mast:poss:15,jean:reve:14} are also not adapted to our problem.

Here again, we assume that the hypotheses \ref{A1} to \ref{A3}, \ref{C1_concave} to \ref{C4_tech} and \ref{H_diff} hold. Recall that from Corollary \ref{lem:integrability_f_H_0}, $\varpi=F^H(\cdot,0,0)$ belongs to $\bL^p(\Omega \times [0,T])$ for any $p>1$. From Lemmata \ref{lem:bounded_cond} and \ref{lem:bound_coef_kappa_2}, $\kappa^1$ and $\kappa^2$ are non-negative and bounded processes. Finally from Lemma \ref{lem:expo_moment_1} and the Novikov's criterion, the martingale defined by \eqref{eq:loc_mart}
\begin{equation*} 
\cE(Z^A)_t = \exp \left( \int_{\widehat T}^t \kappa^1_s \frac{Z^A_s}{A_s} dW_s + \frac{1}{2} \int_{\widehat T}^t (\kappa^1_s)^2 \left( \frac{Z^A_s}{A_s} \right)^2 ds \right),\qquad t \in [\widehat T,T],
\end{equation*}
is uniformly integrable. Using Girsanov's theorem there exists a probability measure $\bQ$ equivalent to $\bP$ such that  $W^\bQ=W- \int \kappa^1_s \frac{Z^A_s}{A_s} ds $ is a Brownian motion under $\bQ$. 
Under $\bQ$, the BSDE \eqref{eq:BSDE_H} becomes
$$H_t = \int_t^T \widetilde{ F^H}(s,H_s)ds- \int_t^T Z^H_s dW^\bQ_s,$$
with
$$\widetilde{ F^H}(t,h)= F(t,h) + \varpi_t + \frac{ \kappa^1_t \kappa^2_t}{2}  \left( \frac{Z^A_t}{A_t} \right)^2 h .$$
To lighten the notations, the generator and the Brownian motion under $\bQ$ are still denoted $F^H$ and $W$. All expectations are considered under $\bQ$. 
A straightforward modification of the proof of Lemma \ref{lem:expo_moment_1} shows that on some deterministic interval $[\widehat T,T]$, there exists $q>1$ such that 
\begin{equation} \label{eq:exponen_moment_Z^A_bis}
\bE \left[ \exp \left( q \int_{\widehat T}^T \dfrac{ \kappa^1_s \kappa^2_s}{2} \left( \frac{Z^A_s}{A_s} \right)^2 ds \right) \right] < +\infty.
\end{equation}

In order to construct the solution, we modify the function $F$. Let us consider $\delta > 0$ and $\eps > 0$ and define
\begin{eqnarray} \label{eq:def_F_delta_eps}
F^{\delta,\eps}(t,h)& = &\frac{-1}{\eta_t \phi'(A_t)} \left[ f(\phi(A_t)-\phi'(A^\delta_t) h) - f(\phi(A_t)) \right]  \mathbf 1_{h\geq 0} +  \frac{\kappa^1_t}{\eta_t A^\eps_t}  h 1_{h\geq 0} 
\end{eqnarray}
with $A^\delta_t = A_{t-\delta}.$
We consider the following BSDE 
\begin{equation} \label{eq:approx_H_BSDE}
H_t = \int_t^T f^{\delta,\eps}(s,H_s) ds - \int_t^T Z^H_s dW_s 
\end{equation}
on the interval $[0,T]$ with 
$$f^{\delta,\eps}(t,h)=   F^{\delta,\eps}(t,h) + \varpi_t +\dfrac{ \kappa^1_t \kappa^2_t}{2}  \left( \dfrac{Z^A_t}{A_t} \right)^2h.$$

We introduce an additional condition on $f$:
\begin{enumerate}[label=\textbf{(C\arabic*)}]
\setcounter{enumi}{4}
\item \label{C5_inte} {\it There exists some $0 \leq \tau < T$ such that for any $r \geq 0$
$$\int_\tau^T \frac{-f \left( \phi(\frac{T-t}{\etamax})+ r\right)}{-\phi'(\frac{T-t}{\etamax})} dt < +\infty.$$
}
\end{enumerate}
Note that the integrand is non-negative. 
From \eqref{eq:bound_A}, since the function $\phi$ is non-increasing, \ref{C5_inte} implies that for any $r \geq 0$ and any $p\geq 1$
$$\bE \left[ \left( \int_\tau^T \frac{f(\phi(A_t) + r)}{\phi'(A_t)} dt \right)^p \right]< +\infty.$$
This hypothesis \ref{C5_inte} depends only on the behavior of $f$ on a neighborhood of $+\infty$. In particular if the function $-f$ is submultiplicative: $-f(x+y) \leq C(-f(x))(-f(y))$ for some fixed constant $C$, then 
$$ \frac{-f \left( \phi(\frac{T-t}{\etamax})+ r\right)}{-\phi'(\frac{T-t}{\etamax})} \leq C (-f( r))  \frac{-f \left( \phi(\frac{T-t}{\etamax})\right)}{-\phi'(\frac{T-t}{\etamax})}  \leq C(-f(r)) .$$
\begin{Rem}[Comments on \ref{C5_inte}] \label{rem:comment_C2}
In the section \ref{ssect:examples}, all functions are submultiplicative (and thus {\rm \ref{C5_inte}} holds), except $f(y) = -\exp(ay^2)$. Nevertheless for this case
\begin{eqnarray*}
\frac{-f(\phi(\frac{T-t}{\etamax}) + r)}{-\phi'(\frac{T-t}{\etamax})} & \leq & C\exp(ar^2) \exp (2ar \phi(A_t)) = \exp(ar^2) \exp(2arG^{-1}(\dfrac{T-t}{\etamax})).
\end{eqnarray*}
And using \eqref{eq:bound_A}
\begin{eqnarray*}
\int_\tau^T \frac{-f(\phi(\frac{T-t}{\etamax}) + r)}{-\phi'(\frac{T-t}{\etamax})} dt 
& \leq & \exp(ar^2) \int_\tau^T \exp \left( 2arG^{-1} \left(\dfrac{T-t}{\etamax}\right) \right) dt \\
& \leq &  \etamax \exp(ar^2)\int_{\zeta}^\infty \exp(2arz)\exp(-az^2) dz < +\infty.
\end{eqnarray*}
Thereby {\rm \ref{C5_inte}} holds also in this case. 
\end{Rem}

\begin{Lem} \label{lem:existence_H_delta_eps}
Assume that {\rm \ref{C5_inte}} holds. 
Define 
$$\mu^\eps_t  = \int_\tau^t \left[ \frac{\kappa^1_s}{\eta_s A^\eps_s} + \dfrac{ \kappa^1_t \kappa^2_t}{2}  \left( \dfrac{Z^A_t}{A_t} \right)^2\right] ds.$$
Then there exists a unique solution $(H^{\delta,\eps},Z^{H,\delta,\eps})$ to the BSDE \eqref{eq:approx_H_BSDE} such that  a.s. for all $t\in[\tau,T]$
$$|H^{\delta,\eps}_t | \leq \bE \left[ \int_t^T e^{\mu^\eps_s-\mu^\eps_t} |\varpi| ds \bigg| \cF_t \right],$$
and for any $p \in (0,q)$
$$ \bE \left[ \sup_{s\in [\tau,T]} \left| e^{\mu^\eps_s} H^{\delta,\eps}_s \right|^p + \left(\int_\tau^T e^{2\mu^\eps_s} |Z^{H,\delta,\eps}_s|^2 ds \right)^{p/2}  \right]  \leq C_q   \bE \left[ \left(  \int_\tau^T e^{\mu^\eps_s} |\varpi_s| ds \right)^p \right].$$
Finally a.s. for any $t \in [\tau,T]$, $H^{\delta,\eps}_t \geq 0$.
\end{Lem}
\begin{proof}
Let us check that all conditions of \cite[Proposition 5.24]{pard:rasc:14} hold (we keep also the same notations). First we have for all $(h,h')$
$$(h-h')(f^{\delta,\eps}(t,h)-f^{\delta,\eps}(t,h')) \leq\left[ \frac{\kappa^1_t}{\eta_t A^\eps_t} + \dfrac{ \kappa^1_t \kappa^2_t}{2}  \left( \dfrac{Z^A_t}{A_t} \right)^2 \right] (h-h')^2$$
Moreover if $|h|\leq r$
\begin{eqnarray*}
|f^{\delta,\eps}(t,h)| & \leq & |\varpi_t| + \left[ \frac{\kappa^1_t}{\eta_t A^\eps_t} + \dfrac{ \kappa^1_t \kappa^2_t}{2}  \left( \dfrac{Z^A_t}{A_t} \right)^2 \right]  r - \frac{f(\phi(A_t)-\phi'(A^\delta_t) r)}{- \eta_t \phi'(A_t)} \\
& \leq & |\varpi_t| + \left[ \frac{\kappa^1_t}{\eta_t A^\eps_t} + \dfrac{ \kappa^1_t \kappa^2_t}{2}  \left( \dfrac{Z^A_t}{A_t} \right)^2\right]  r - \frac{f(\phi(A_t)-\phi'(\delta/\etamax)  r)}{- \eta_t \phi'(A_t)}  =  \Phi^\sharp_r(t).
\end{eqnarray*}
Note that $\frac{\kappa^1}{\eta A^\eps}$ is bounded on $[0,T]$ and that $\varpi$ belongs to $\bL^p(\Omega\times [0,T])$ for any $p > 1$.
From Condition \ref{C5_inte}, using Estimate \eqref{eq:exponen_moment_Z^A_bis}, we deduce that $\Phi^\sharp_r$ also belongs to $\bL^p(\Omega\times [0,T])$ for any $p > 1$. Hence for $p<q$, using H\"older's inequality
\begin{align*}
&\bE \left[ \left( \int_\tau^T e^{\mu^\eps_s}  \Phi^\sharp_r(s) ds \right)^p \right]  \leq \bE \left[e^{p\mu^\eps_T}  \left( \int_\tau^T  \Phi^\sharp_r(s) ds \right)^p \right] \\
& \qquad \leq\left(  \bE \left[e^{q\mu^\eps_T} \right]\right)^{\frac{p}{q}} \left( \bE \left[ \left( \int_\tau^T  \Phi^\sharp_r(s) ds \right)^{\frac{pq}{q-p}} \right]\right)^{\frac{q-p}{q}} < +\infty.
\end{align*}
Using \cite[Proposition 5.24]{pard:rasc:14}, we deduce that there exists a unique solution $(H^{\delta,\eps},Z^{H,\delta,\eps})$ satisfying the desired estimate. The non-negativity of $H^{\delta,\eps}$ is obtained by comparison principle (\cite[Proposition 5.34]{pard:rasc:14}) since $F^H(\cdot,0) \geq 0$ a.s. 
\end{proof}

\vspace{0.5cm}
Let us begin with an a priori estimate on $H^{\delta,\eps}$. Recall that the function $\vartheta$ is defined just before Lemma \ref{lem:gene_a_priori_estim}.
\begin{Lem} \label{lem:a_priori_estim_H}
For all $t \in [\tau,T)$,
\begin{equation} \label{eq:a_priori_estimate_H}
0 \leq H^{\delta,\eps}_t \leq \frac{\vartheta(T-t)}{-\phi'(A_t)}.
\end{equation}
In particular $H^{\delta,\eps}$ is bounded on any interval $[\tau,T-\theta]$, $0 < \theta < T-\tau$.
\end{Lem}
\begin{proof}
For fixed $\delta$ and $\eps$, the dynamics of $Y^{\delta,\eps}_t= \phi(A_t) -\phi'(A_t) H^{\delta,\eps}_t$ is given by:
\begin{align*}
&-dY^{\delta,\eps}_t =   \phi'(A_t) \frac{1}{\eta_t } dt+ \phi'(A_t) Z^A_t dW_t  -  \frac{1}{2} \phi^{(2)}(A_t) (Z^A_t)^2 dt \\
&\qquad -   \phi^{(2)}(A_t) \frac{H^{\delta,\eps}_t}{\eta_t } dt- \phi^{(2)}(A_t) Z^A_t H^{\delta,\eps}_t dW_t  +  \frac{1}{2} \phi^{(3)}(A_t) (Z^A_t)^2 H^{\delta,\eps}_t dt\\
&\qquad - \phi'(A_t)  F^{\delta,\eps}(t, H^{\delta,\eps}_t) dt - \phi'(A_t) F^H(t,0) dt - \phi'(A_t) \alpha_t H^{\delta,\eps}_t dt\\
&\qquad  -  \phi'(A_t) Z^{H,\delta,\eps}_t dW_t.
\end{align*}
Recall that 
$$- \phi'(A_t) F^H(t,0) = \lambda_t - \phi'(A_t) \frac{\kappa^1_t}{2} \frac{(Z^A_t)^2}{A_t} = \lambda_t +\frac{\phi^{(2)}(A_t)}{2} (Z^A_t)^2.$$
Since $F^{\delta,\eps}$ is given by \eqref{eq:def_F_delta_eps}, 
\begin{eqnarray*} 
- \phi'(A_t) F^{\delta,\eps}(t,H^{\delta,\eps}_t)& = &\frac{1}{\eta_t} \left[ f(Y^{\delta,\eps}_t) - \phi'(A_t) \right] +  \frac{\phi^{(2)}(A_t)A_t}{\eta_t A^\eps_t}  H^{\delta,\eps}_t .
\end{eqnarray*}
And
$$- \phi'(A_t)\alpha_t =  \dfrac{ \kappa^1_t \kappa^2_t}{2}  \left( \dfrac{Z^A_t}{A_t} \right)^2 =- \frac{1}{2} \phi^{(3)}(A_t) (Z^A_t)^2.$$
Therefore we obtain
\begin{eqnarray*}
-dY^{\delta,\eps}_t & =  &\left[ \frac{1}{\eta_t}  f(Y^{\delta,\eps}_t) +\lambda_t \right] dt  - \phi^{(2)}(A_t) \frac{H^{\delta,\eps}_t}{\eta_t } \left[1- \frac{A_t}{A^\eps_t} \right] dt \\
&+& \left[ \phi'(A_t) Z^A_t - \phi^{(2)}(A_t) Z^A_t H^{\delta,\eps}_t-  \phi'(A_t) Z^{H,\delta,\eps}_t \right] dW_t.
\end{eqnarray*}
In other words the process $Y^{\delta,\eps}$ satisfies the BSDE:
\begin{eqnarray*}
Y^{\delta,\eps}_t & = & Y^{\delta,\eps}_{T-\theta} +\int_t^{T-\theta} \left[ \lambda_s + \frac{1}{\eta_s} f(Y^{\delta,\eps}) \right]ds - \int_t^{T-\theta} \Theta_s ds  -  \int_t^{T-\theta}  Z^{Y^{\delta,\eps}}_s dW_s,
\end{eqnarray*}
with a non negative $\Theta$. Using Lemma \ref{lem:gene_a_priori_estim}, we deduce that 
$$\forall t \in [\tau,T), \quad Y^{\delta,\eps}_t \leq \vartheta \left( T -t \right) .$$
This leads to the conclusion of the Lemma.
\end{proof}

%
Now by the comparison principle, for a fixed $\delta > 0$, since $\kappa^1_t/\eta_t \geq 0$, $(H^{\delta,\eps}_t, \ \eps > 0)$ is a increasing sequence when $\eps$ decreases to zero, and for a fixed $\eps > 0$, $(H^{\delta,\eps}_t, \ \delta > 0)$ is a decreasing sequence when $\delta$ decreases to zero. Thereby for any $\eps_1 < \eps_2$ and $\delta_1 < \delta_2 \leq \delta$ for some $\delta > 0$, we have the following inequalities: a.s. 
$$0 \leq H^{\delta_1,\eps_1}_t \leq H^{\delta_2,\eps_1}_t \leq H^{\delta}_t,$$
and
$$0 \leq H^{\delta_1,\eps_2}_t \leq H^{\delta_1,\eps_1}_t \leq H^{\delta}_t,$$
where
\begin{equation} \label{eq:def_H_delta}
H^\delta_t = \lim_{\eps \downarrow 0} H^{\delta,\eps}_t.
\end{equation}
Note that $H^\delta$ also satisfies \eqref{eq:a_priori_estimate_H}. Now for a fixed $\eps > 0$, we define
\begin{equation} \label{eq:def_H_eps}
H^\eps_t = \lim_{\delta \downarrow 0} H^{\delta,\eps}_t,
\end{equation}
and
\begin{equation} \label{eq:def_H}
H_t = \lim_{\eps \downarrow 0} H^{\eps}_t.
\end{equation}
Since $H^{\delta,\eps}_T = 0$ a.s., we have immediately that a.s. $H_T = 0$ and for all $t\in [\tau,T]$, $H_t \geq 0$.

\begin{Prop}  \label{prop:existence_H_gene}
There exists $Z^H \in \bH^p(0,T-)$, such that the couple $(H,Z^H)$ solves the BSDE \eqref{eq:BSDE_H} with generator $F^H$ on the interval $[0,T-\theta]$ for any $0< \theta < T$. 
\end{Prop}
\begin{proof}
Let us define $c_p = p(1\wedge (p-1))$. First we work on the interval $[\tau,T]$.

\noindent {\it Step 1.} Given $\eps_1 < \eps_2$ and $\delta_1 < \delta_2$, applying It\^o's formula to $\Delta H = H^{\delta_1,\eps_1} - H^{\delta_2,\eps_2}$ on the interval $[t,T-\theta]$, $\theta > 0$, leads to:
\begin{eqnarray*}
&&e^{p\widehat \mu_t} |\Delta H_t|^p + \frac{c_p}{2} \int_t^{T-\theta}e^{p\widehat \mu_s} |\Delta H_s|^{p-2} \mathbf{1}_{\Delta H_s \neq 0} |\Delta Z^H_s|^2 ds \\
&& \quad \leq e^{p\widehat \mu_T} |\Delta H_{T-\theta}|^p \\
&& \qquad + p \int_t^{T-\theta}e^{p\widehat \mu_s} |\Delta H_s|^{p-2} \mathbf{1}_{\Delta H_s \neq 0} \Delta H_s (f^{\delta_1,\eps_1}(s,H^{\delta_1,\eps_1}_s) - f^{\delta_2,\eps_2}(s,H^{\delta_2,\eps_2}_s)) ds \\
&& \qquad - p \int_t^{T-\theta} \left(\alpha_s + \dfrac{\kappa^1_s}{\eta_s \theta}  + 2\frac{p-1}{p} \right) e^{p\widehat \mu_s}|\Delta H_s|^p ds  - p \int_t^{T-\theta} e^{p\widehat \mu_s}|\Delta H_s|^{p-2} \mathbf{1}_{\Delta H_s \neq 0} \Delta H_s\Delta Z^H_s dW_s
\end{eqnarray*}
with $\Delta Z^H = Z^{H,\delta_1,\eps_1} - Z^{H,\delta_2,\eps_2}$ and 
$$\displaystyle \widehat \mu_t = \int_0^t \left( \alpha_s + \dfrac{\kappa^1_s\etamax}{\eta_s \theta} +2 \frac{p-1}{p} \right) ds.$$ 
Remark that from the monotonicity of $f$:
\begin{eqnarray*}
&&\Delta H_s(f^{\delta_1,\eps_1}(s,H^{\delta_1,\eps_1}_s) - f^{\delta_2,\eps_2}(s,H^{\delta_2,\eps_2}_s)) =\Delta H_s \left[ \frac{\kappa^1_s}{\eta_s A^{\eps_1}_s} H^{\delta_1,\eps_1}_s -\frac{\kappa^1_s}{\eta_s A^{\eps_2}_s} H^{\delta_2,\eps_2}_s \right] \\
&&\qquad + \frac{\Delta H_s}{-\phi'(A_s)\eta_s} \left[  f(\phi(A_s) - \phi'(A^{\delta_1}_s) H^{\delta_1,\eps_1}_s) - f(\phi(A_s) - \phi'(A^{\delta_2}_s) H^{\delta_2,\eps_2}_s) \right] +\alpha_s (\Delta H_s)^2  \\
&& \quad \leq  \left[ \frac{\kappa^1_s}{\eta_s A^{\eps_1}_s}+ \alpha_s  \right]   (\Delta H_s)^2 + \frac{\kappa^1_s}{\eta_s }  \Delta H_s H^{\delta_2,\eps_2}_s \left[ \frac{1}{A^{\eps_1}_s} -\frac{1}{A^{\eps_2}_s}\right]  \\
&&\qquad + \frac{\Delta H_s}{-\phi'(A_s)\eta_s}  \left[  f(\phi(A_s) - \phi'(A^{\delta_1}_s) H^{\delta_2,\eps_2}_s) - f(\phi(A_s) - \phi'(A^{\delta_2}_s) H^{\delta_2,\eps_2}_s) \right] .
\end{eqnarray*}
We denote 
$$\Delta f_s =   f(\phi(A_s) - \phi'(A^{\delta_1}_s) H^{\delta_2,\eps_2}_s) - f(\phi(A_s) - \phi'(A^{\delta_2}_s) H^{\delta_2,\eps_2}_s).$$
We deduce that 
\begin{eqnarray*}
&&e^{p\widehat \mu_t} |\Delta H_t|^p + \frac{c_p}{2} \int_t^{T-\theta}e^{p\widehat \mu_s} |\Delta H_s|^{p-2} \mathbf{1}_{\Delta H_s \neq 0} |\Delta Z^H_s|^2 ds \\
&& \quad \leq e^{p\widehat \mu_{T-\theta}}|\Delta H_{T-\theta}|^p + p \int_t^{T-\theta} e^{p\widehat \mu_s}\left[\frac{\kappa^1_s}{\eta_s A^{\eps_1}_s}- \dfrac{\kappa^1_s\etamax }{\eta_s \theta}  -2\frac{p-1}{p}\right] |\Delta H_s|^{p} ds \\
&& \qquad + p \int_t^{T-\theta}e^{p\widehat \mu_s} |\Delta H_s|^{p-2} \mathbf{1}_{\Delta H_s \neq 0} \frac{\kappa^1_s}{\eta_s }  \Delta H_s H^{\delta_2,\eps_2}_s \left( \frac{1}{A^{\eps_1}_s} -\frac{1}{A^{\eps_2}_s}\right) ds \\ 
&&\qquad + p \int_t^{T-\theta}e^{p\widehat \mu_s} |\Delta H_s|^{p-2} \mathbf{1}_{\Delta H_s \neq 0}  \frac{\Delta H_s}{-\phi'(A_s)\eta_s} \Delta f_s ds \\
&& \qquad - p \int_t^{T-\theta}e^{p\widehat \mu_s} |\Delta H_s|^{p-2} \mathbf{1}_{\Delta H_s \neq 0} \Delta H_s\Delta Z^H_s dW_s.
\end{eqnarray*}
Note that for any $t\leq T-\theta$
$$A^{\eps_1}_t \geq \frac{T-t+\eps_1}{\etamax} \geq \frac{\theta+\eps_1}{\etamax} \geq \frac{\theta}{\etamax} .$$
By Young's inequality
\begin{eqnarray*}
&& p\int_t^{T-\theta}e^{p\widehat \mu_s} |\Delta H_s|^{p-2} \mathbf{1}_{\Delta H_s \neq 0}  \frac{\kappa^1_s}{\eta_s }  \Delta H_s H^{\delta_2,\eps_2}_s \left( \frac{1}{A^{\eps_1}_s} -\frac{1}{A^{\eps_2}_s}\right)ds \\
&&\quad \leq (p-1) \int_t^{T-\theta}e^{p\widehat \mu_s} |\Delta H_s|^{p}ds +  \int_t^{T-\theta}e^{p\widehat \mu_s} \left[H^{\delta_2,\eps_2}_s \frac{\kappa^1_s}{\eta_s }   \left( \frac{1}{A^{\eps_1}_s} -\frac{1}{A^{\eps_2}_s}\right) \right]^p ds 
\end{eqnarray*}
and
\begin{eqnarray*}
&& p\int_t^{T-\theta}e^{p\widehat \mu_s} |\Delta H_s|^{p-2} \mathbf{1}_{\Delta H_s \neq 0} \frac{\Delta H_s}{-\phi'(A_s)\eta_s} \Delta f_sds \\
&&\quad \leq (p-1) \int_t^{T-\theta}e^{p\widehat \mu_s} |\Delta H_s|^{p}ds  +  \int_t^{T-\theta}e^{p\widehat \mu_s} \left[\frac{1}{-\phi'(A_s)\eta_s}\left|  \Delta f_s \right| \right]^p ds .
\end{eqnarray*}
Recall that $\kappa^1$, $1/\eta$ are bounded on $[0,T]$ whereas $1/\phi'(A)$ is bounded on $[0,T-\theta]$. From Estimate \eqref{eq:a_priori_estimate_H}, on the interval $[0,T-\theta]$, $H^{\eps_2,\delta_2}$ is also bounded and we have:
\begin{eqnarray} \nonumber 
&&e^{p\widehat \mu_t} |\Delta H_t|^p + \frac{c_p}{2} \int_t^{T-\theta}e^{p\widehat \mu_s} |\Delta H_s|^{p-2} \mathbf{1}_{\Delta H_s \neq 0} |\Delta Z^H_s|^2 ds \\ \nonumber 
&& \quad \leq e^{p\widehat \mu_{T-\theta}}|\Delta H_{T-\theta}|^p   + C \int_t^{T-\theta} \left[ \left| \Delta f_s \right|^p +  \left( \frac{1}{A^{\eps_1}_s} -\frac{1}{A^{\eps_2}_s}\right)^p \right] ds\\ \label{eq:estim_diff_conv}
&& \qquad - p \int_t^{T-\theta}e^{p\widehat \mu_s} |\Delta H_s|^{p-2} \mathbf{1}_{\Delta H_s \neq 0} \Delta H_s\Delta Z^H_s dW_s.
\end{eqnarray}
The constant $C$ depends on all bounds of our coefficients and on $\theta$. This constant explodes when $\theta$ goes to zero.

\noindent {\it Step 2.} Let us fix $\eps=\eps_1=\eps_2 > 0$. Since for $ \delta_2 \leq \delta$, $H^{\delta_2,\eps}\leq H^{\delta}$ and $H^\delta$ satisfies the estimate \eqref{eq:a_priori_estimate_H}, using the dominated convergence theorem 
$$\bE  \int_\tau^{T-\theta} e^{p\widehat \mu_s} \left|   f(\phi(A_s) - \phi'(A^{\delta_1}_s) H^{\delta_2,\eps_2}_s) - f(\phi(A_s) - \phi'(A^{\delta_2}_s) H^{\delta_2,\eps_2}_s) \right|^p ds \to 0,$$
as $\delta_1$ and $\delta_2$  tend to zero. Therefore using \eqref{eq:estim_diff_conv} and taking the expectation we deduce that 
$$\bE \int_\tau^{T-\theta}e^{p\widehat \mu_s} |\Delta H_s|^{p-2} \mathbf{1}_{\Delta H_s \neq 0} |\Delta Z^H_s|^2 ds $$
tends to zero when $\delta_1$ and $\delta_2$ go to zero. Moreover remark that if 
$$\Lambda_t = \int_\tau^t e^{p\widehat \mu_s} |\Delta H_s|^{p-2} \mathbf{1}_{\Delta H_s \neq 0} \Delta H_s\Delta Z^H_sdW_s ,$$
then the bracket $[\Lambda]_{T-\theta}^{1/2}$ can be handled as in \cite{bria:dely:hu:03}: for any $C> 0$
\begin{eqnarray*}
\bE \left(  [\Lambda]_{T-\theta}^{1/2} \right) & \leq & \frac{C}{2} \bE \left( \sup_{t\in [\tau,T-\theta]} e^{p\widehat \mu_{t}}|\Delta H_{t}|^p \right) \\
& + & \frac{1}{2C} \bE \left( \int_\tau^{T-\theta} e^{p\widehat \mu_s} |\Delta H_s|^{p-2} \mathbf{1}_{\Delta H_s \neq 0} |\Delta Z^H_s|^2 ds \right).
\end{eqnarray*}
Thereby $(H^{\delta,\eps}, \ \delta > 0)$ is a Cauchy sequence:
$$\bE \left( \sup_{t\in [\tau,T-\theta]} e^{p\widehat \mu_{t}}|\Delta H_{t}|^p \right) \to 0$$
as $\delta_1$ and $\delta_2$ tend to zero. Finally 
\begin{eqnarray*} \nonumber
&& \bE \left( \int_\tau^{T-\theta} e^{2\widehat \mu_s}|\Delta Z^H_s|^2 ds \right)^{p/2}  =  \bE \left( \int_\tau^{T-\theta}e^{2\widehat \mu_s} \left( \Delta H_{s} \right)^{2-p}\left( \Delta H_{s} \right)^{p-2}\mathbf{1}_{\Delta H_s \neq 0}  |\Delta Z^H_s|^2ds \right)^{p/2} \\ \nonumber
&& \quad \leq \bE \left[ \left( \sup_{t\in [\tau,T-\theta]} e^{\widehat \mu_{t}}|\Delta H_{t}|\right)^{p(2-p)/2} \left(\int_\tau^{T-\theta} e^{p\widehat \mu_s} \left( \Delta H_{s} \right)^{p-2}\mathbf{1}_{\Delta H_s \neq 0}  |\Delta Z^H_s|^2 ds \right)^{p/2}\right]\\ \nonumber
&& \quad \leq \left\{\bE  \left( \sup_{t\in [\tau,T-\theta]} e^{p\widehat \mu_{t}}|\Delta H_{t}|^p \right)\right\}^{(2-p)/2}  \left\{\bE \int_\tau^{T-\theta} e^{p\widehat \mu_s} \left( \Delta H_{s} \right)^{p-2}\mathbf{1}_{\Delta H_s \neq 0}  |\Delta Z^H_s|^2 ds \right\}^{p/2} \\ \label{eq:trick_control_mart_part}
&& \quad  \leq \frac{2-p}{2} \bE \left[ \sup_{t\in [\tau,T-\theta]} e^{p\widehat \mu_{t}}|\Delta H_{t}|^p\right] + \frac{p}{2} \bE \int_\tau^{T-\theta} e^{p\widehat \mu_s} \left( \Delta H_{s} \right)^{p-2}\mathbf{1}_{\Delta H_s \neq 0}  |\Delta Z^H_s|^2  ds
\end{eqnarray*}
where we have used H\"older's and Young's inequality with $ \frac{2-p}{2} + \frac{p}{2}=1$. 

Hence we obtain that $(H^{\delta,\eps},Z^{H,\delta,\eps})$ converges in $\bS^p(\tau,T-\theta)$ to some process $(H^\eps,Z^{H,\eps})$. The process $H^\eps$ is non negative and also satisfies the a priori estimate \eqref{eq:a_priori_estimate_H} with $H^\eps_T = 0$, and we have for any $\tau\leq t \leq T-\theta  < T$
\begin{eqnarray}\nonumber
H^\eps_t & = & H^\eps_{T-\theta} + \int_t^{T-\theta} \varpi_s ds  +  \int_t^{T-\theta}  \frac{\kappa^1_s}{\eta_s A^\eps_s} H^\eps_s ds- \int_t^{T-\theta} Z^{H,\eps}_s dW_s \\ \label{eq:BSDE_H_eps}
 &+ & \int_t^{T-\theta} \frac{1}{-\phi'(A_s) \eta_s}\left[  f(\phi(A_s) -\phi'(A_s) H^\eps_s) - f (\phi(A_s)) \right] ds .
\end{eqnarray}

\noindent {\it Step 3.} Let us prove the convergence of $(H^\eps,Z^{H,\eps})$ when $\eps$ tends to zero. The arguments are almost the same as in the second step. Indeed the formula \eqref{eq:estim_diff_conv} becomes:
\begin{eqnarray*} \nonumber 
&&e^{p\widehat \mu_t} |\Delta H_t|^p + \frac{c_p}{2} \int_t^{T-\theta}e^{p\widehat \mu_s} |\Delta H_s|^{p-2} \mathbf{1}_{\Delta H_s \neq 0} |\Delta Z^H_s|^2 ds \\ \nonumber 
&& \quad \leq e^{p\widehat \mu_{T-\theta}}|\Delta H_{T-\theta}|^p   + C \int_t^{T-\theta} \left[ \left( \frac{1}{A^{\eps_1}_s} -\frac{1}{A^{\eps_2}_s}\right)^p \right] ds\\ \nonumber 
&& \qquad - p \int_t^{T-\theta}e^{p\widehat \mu_s} |\Delta H_s|^{p-2} \mathbf{1}_{\Delta H_s \neq 0} \Delta H_s\Delta Z^H_s dW_s.
\end{eqnarray*}
with  $\Delta H = H^{\eps_1} - H^{\eps_2}$ and $\Delta Z^H = Z^{H,\eps_1} - Z^{H,\eps_2}$. The conclusion follows from the same arguments as in step 2.

Hence we have construct $(H,Z^H)$ solution of the BSDE \eqref{eq:BSDE_H} on the interval $[\tau,T)$. Now if we consider the BSDE on $[0,\tau]$
$$H_t = H_{\tau} + \int_t^\tau F^H(s,H_s) ds - \int_t^\tau Z^H_s dW_s,$$
we can apply directly \cite[Proposition 5.24]{pard:rasc:14}. Indeed the terminal value $H_\tau$ is bounded, and the arguments of the proof can be applied. 
\end{proof}

Note that from \eqref{eq:a_priori_estimate_H}  and the arguments to prove \eqref{eq:a_priori_estim_Y_bis}, we obtain that for any $\eps > 0$,  a.s. on $[0,T)$
\begin{equation} \label{eq:upper_Y_estimate}
0\leq \phi(A_t) - \phi'(A_t) H_t \leq \phi \left( \frac{T-t}{(1+\eps)\etamax}\right) .
\end{equation}
and 
\begin{equation} \label{eq:H_estimate}
0\leq  H_t \leq C \frac{\phi \left( \frac{T-t}{(1+\eps)\etamax}\right)}{-\phi'\left( \frac{T-t}{\etamax}\right)} \leq C \frac{\phi \left( \frac{T-t}{(1+\eps)\etamax}\right)}{-\phi'\left( \frac{T-t}{(1+\eps)\etamax}\right)} = C \frac{\phi \left( \frac{T-t}{(1+\eps)\etamax}\right)}{-f\left(\phi\left( \frac{T-t}{(1+\eps)\etamax}\right)\right)}. 
\end{equation}
Since $f$ is a non positive and non increasing function, from \ref{A3}, the function $y\mapsto -1/f(y)$ is an integrable, non-negative and non increasing function. Thereby we know that (see \cite[Section 178]{hardy:08})
$$\lim_{y \to + \infty} \frac{y}{-f(y)} = 0.$$ 
Arguing as in the proof of Proposition \ref{prop:existence_sol_BSDE_sing_gen}, we obtain that this solution $(H,Z^H)$ is a solution of the BSDE \eqref{eq:BSDE_H} (in the sense of Definition \ref{def:sing_gene_sol}). To summarize
\begin{Thm} \label{thm:exist_H}
Assume that {\rm \ref{A1}} to {\rm \ref{A3}}, {\rm \ref{C1_concave}} to {\rm \ref{C5_inte}}  and {\rm \ref{H_diff}} hold. There exists a process $(H,Z^H)$, which is the minimal non-negative solution of the BSDE \eqref{eq:BSDE_H}, that is:
\begin{itemize}
\item $H$ is non negative and essentially bounded: for any $0 \leq t < T$, $0\leq \sup_{s\in[0,t]} H_s <+\infty$ a.s. and 
$$\bE \int_0^T | F^H(s,H_s,Z^H_s)| ds < +\infty.$$
\item The process $Z^H$ belongs to $\bH^1(0,T) \cap \bH^p(0,T-)$ for some $p>1$.
\item For any $0 \leq t \leq T$
$$H_t = \int_t^T F^H(s,H_s,Z^H_s) ds - \int_t^T Z^H_s dW_s.$$
In particular
$$\lim_{t\to T} H_t = 0 = H_T.$$
\item For any  $(\widehat H, \widehat Z)$ solution of the BSDE  \eqref{eq:BSDE_H}, a.s. for any $t \in [0,T]$, $\widehat H_t \geq H_t$. 
\end{itemize}
\end{Thm}
\begin{proof}
The only thing to prove is the minimality. Let $(\widehat H, \widehat Z)$ be another solution of the BSDE  \eqref{eq:BSDE_H}. The first part of the proof of Proposition \ref{prop:sing_gene_uniqueness} shows that $\widehat H$ is a non-negative process. 

Now for $\eps > 0$, from \eqref{eq:BSDE_H_eps}, the process $\Delta H = \widehat H - H^\eps$ satisfies for any $\theta > 0$:
\begin{eqnarray*}
\Delta H _t & = & \Delta H_{T-\theta} + \int_t^{T-\theta} \left[ \frac{\kappa^1_s}{\eta_s}\left( \frac{1}{A_s} -\frac{1}{A^\eps_s} \right)\widehat H_s +\left( \frac{\kappa^1_s}{\eta_s A^\eps_s} +\dfrac{ \kappa^1_s \kappa^2_s}{2}  \left( \dfrac{Z^A_s}{A_s} \right)^2\right) \Delta H_s  \right] ds \\
& + &  \int_t^{T-\theta} \frac{1}{-\phi'(A_s) \eta_s }\left[ f(\phi(A_s) - \phi'(A_s) \widehat H_s) - f(\phi(A_s) - \phi'(A_s) H^\eps_s) \right] ds \\
& - &  \int_t^{T-\theta} (\widehat Z_s - Z^\eps_s) dW^{\bQ}_s.
\end{eqnarray*}
If 
$$\nu_s =  \frac{1}{-\phi'(A_s) \eta_s }\left[ f(\phi(A_s) - \phi'(A_s) \widehat H_s) - f(\phi(A_s) - \phi'(A_s) H^\eps_s) \right]  \frac{1}{\Delta H_s} \mathbf{1}_{\Delta H_s \neq 0},$$
then this process $\nu$ is bounded from above by zero since $f$ is monotone. Thus if
$$\Gamma_{t,s} = \exp \left( \int_t^s\left( \frac{\kappa^1_u}{\eta_u A^\eps_u} + \dfrac{ \kappa^1_u \kappa^2_u}{2}  \left( \dfrac{Z^A_u}{A_u} \right)^2\right) du \right),$$
by standard arguments concerning linear BSDE (see \cite[Lemma 10]{krus:popi:16} or \cite[Proposition 5.31]{pard:rasc:14}), we have:
$$\Delta H _t  = \bE^{\bQ} \left[ \Delta H_{T-\theta} \Gamma_{t,T-\theta} + \int_t^{T-\theta}  \frac{\kappa^1_s}{\eta_s}\left( \frac{1}{A_s} -\frac{1}{A^\eps_s} \right)\widehat H_s\Gamma_{t,s} ds \bigg| \cF_t\right] \geq \bE^{\bQ} \left[ \Delta H_{T-\theta} \Gamma_{t,T-\theta} \bigg| \cF_t\right]. $$
By Fatou's lemma, letting $\theta$ going to zero, we obtain that for any $\eps > 0$, $\widehat H_t \geq H^\eps_t$. Hence the minimality of $H$ is proved.
\end{proof}

The arguments developed above, with straightforward modifications, show the next result. 
\begin{Prop} \label{prop:exist_H}
Assume that {\rm \ref{A1}} to {\rm \ref{A3}}, {\rm \ref{C2_tech}} and {\rm \ref{C5_inte}} hold. There exists a process $(\widehat H,Z^{\widehat H})$, which is the minimal non-negative solution of the BSDE \eqref{eq:BSDE_H_general_case}. 
\end{Prop}

Concerning uniqueness, from Proposition \ref{prop:sing_gene_uniqueness}, the minimal solution of Theorem \ref{thm:exist_H} is unique. However in the preceding Proposition, the process $(\widehat H,Z^{\widehat H})$ solves a BSDE with singular driver in the sense of \cite{jean:reve:14}. As mentioned in \cite[Proposition 3.1]{jean:reve:14}, uniqueness is not an obvious property for such kind of BSDEs. In our case assume that $(\widetilde H, \widetilde Z)$ be another non negative solution of the BSDE \eqref{eq:BSDE_H_general_case}. Then 
\begin{eqnarray*}
\Delta H _t  &= & \int_t^{T} \left[ f^{\widehat H}(s,\widetilde H_s) - f^{\widehat H}(s,\widehat H_s) \right] ds - \int_t^{T} (\widetilde Z_s - Z^{\widehat H}_s) dW_s \\
& = & \int_t^{T} \nu_s \Delta H_s ds- \int_t^{T} \Delta Z_s  dW_s .
\end{eqnarray*}
Hence we have a linear BSDE with singular generator with 
$$\nu_s =  \frac{\kappa^3_s}{T-s} + \frac{1}{-\psius(s) \eta_s } \left[ f(\phi(A_s) - \psius(s) \widetilde H_s) - f(\phi(A_s) - \psius(s) \widehat H_s)\right] \frac{1}{\Delta H_s} \mathbf{1}_{\Delta H_s \neq 0}.$$
Nevertheless we cannot apply the result in \cite[Propositions 3.1 and 3.5]{jean:reve:14}, since we don't know the sign of the drift $\nu$. The concavity assumption \ref{C1_concave} has a key role to obtain the sign and thus uniqueness. 

\begin{Rem}
In the proof of Proposition \ref{prop:existence_H_gene}, we can define $\displaystyle H^{\natural}_t = \lim_{\delta \downarrow 0} H^{\delta}_t$. Then $H_t \leq H^\natural_t$.  We can prove that $H^\sharp$ also satisfies the BSDE \eqref{eq:BSDE_H} (or the BSDE \eqref{eq:BSDE_H_general_case}). Since the BSDE \eqref{eq:BSDE_H} has a unique solution, $H=H^\natural$. However for BSDE \eqref{eq:BSDE_H_general_case}, it seems difficult to prove that these two processes are equal. In other words, as remarked above, we don't have any comparison or uniqueness result concerning the BSDE \eqref{eq:BSDE_H_general_case}.
\end{Rem}

\bibliographystyle{plain}
\bibliography{biblio}

\end{document}